\documentclass[10pt, a4paper]{article}
\usepackage[utf8]{inputenc}
\usepackage{amsfonts}
\usepackage{amsmath}
\usepackage{amsthm}
\usepackage{amssymb}
\usepackage{mathrsfs}
\usepackage{xcolor}
\usepackage[colorlinks]{hyperref} % hyperlinks
\hypersetup{citecolor=blue}
\usepackage{cleveref}
\usepackage[round]{natbib}
\usepackage{xspace}
\usepackage{comment}
\usepackage{graphicx}

\crefname{appendix}{appendix}{appendices}
\Crefname{appendix}{Appendix}{Appendices}

\newcommand{\comm}[1]{\textcolor{orange}{#1}}

\newcommand{\A}{\mathcal A}
\newcommand{\B}{\mathcal B}
\newcommand{\bHhP}{\bar\Hh_{\Phi,S}}

\newcommand{\E}{\mathbb E}
\newcommand{\e}{e}
\newcommand{\Ee}{\mathcal E}
\newcommand{\EeP}{\Ee_{\LP}}
\newcommand{\Ell}{\mathcal L}
\newcommand{\Hh}{\mathcal H}
\newcommand{\HhP}{\Hh_{\Phi,S}}
\newcommand{\I}{\mathcal I}
\newcommand{\Id}{\mathrm{Id}}
\newcommand{\LP}{\Lambda_{\Phi,S}}
\newcommand{\LPmax}{{\LP^\star}}

\newcommand{\one}{\mathbf{1}}

\newcommand{\PR}{\mathscr{P}}
\newcommand{\R}{\mathbb R}
\newcommand{\ROOT}{\circledast}
\newcommand{\sP}{\sigma_{\Phi,S}}
\newcommand{\T}{\mathcal T}
\newcommand{\Tt}{\mathbb T}

\newcommand{\V}{\mathcal{V}}
\newcommand{\X}{\mathcal X}
\newcommand{\XP}{\X_{\Phi,S}}

\newcommand{\Z}{\mathcal Z}

\DeclareMathOperator{\aff}{aff}
\DeclareMathOperator{\ch}{ch}
\DeclareMathOperator{\conv}{conv}
\DeclareMathOperator{\pa}{pa}
\DeclareMathOperator{\relint}{relint}

\DeclareMathOperator{\Tr}{Tr}

\newtheorem{lemma}{Lemma}
\newtheorem{proposition}{Proposition}
\newtheorem{theorem}{Theorem}
\newtheorem{corollary}{Corollary}

\newcommand{\reptile}{}
\newtheorem*{repinner}{\reptile}

\newenvironment{reflem}[1]{%
  \renewcommand{\reptile}{Lemma \ref{#1}}%
  \begin{repinner}
}{%
  \end{repinner}
}

\author{Eugenio Clerico\footnote{Correspondence to eugenio.clerico@gmail.com}}
\date{\small Department of Statistics, University of Oxford}
\title{Sequential testing of\\ conditionally constrained hypotheses} 

\begin{document}
\maketitle
\begin{abstract}
We characterise e-processes for testing non-parametric hypotheses that require the conditional mean of a finite-dimensional statistic to lie in a prescribed convex set. Our main result is a complete-class theorem: every such e-process is pointwise dominated by a predictable product of affine one-step e-variables. Consequently, restricting to test supermartingales entails no loss: at any rejection threshold, the dominating process rejects no later on any observation path. We further show that a minimal complete class exists and consists of predictable products of maximal affine one-step e-variables. These results extend previous complete-class results from single-step constrained testing and bounded one-dimensional conditional mean testing to a broader conditional sequential setting.\end{abstract}

\section{Introduction}\label{sec:intro}
Statistical anytime-valid  inference (SAVI) has recently attracted considerable interest as a framework for sequential and adaptive hypothesis testing \citep{shafer2021testing,vovk2021evalues,grunwald2024safe,ramdas2023game,ramdas2024hypothesis}. A central role in SAVI is played by \emph{e-values}, the realisations of non-negative random variables (the so-called called \emph{e-variables}) whose expectation under the null hypothesis is bounded by one. E-values can be viewed as generalisations of likelihood ratios and Bayes factors \citep{shafer2011test,wasserman2020universal}. Large e-values provide evidence against the null hypothesis, as Markov's inequality  turns them into valid tests. Their usefulness is especially clear in sequential settings: if, at each round, one chooses an e-variable conditionally on the past, then their product is a non-negative supermartingale under the null. Ville's inequality \citep{ville1939etude} controls the probability that such  product ever becomes large and therefore yields anytime-valid tests, where the data may be inspected repeatedly and the experiment stopped at a data-dependent time, without invalidating the inference \citep{ramdas2024hypothesis}.

Products of e-values, however, are not the most general objects available for sequential testing. A broader notion is that of an \emph{e-process}: a non-negative random process, adapted to the sequence of observations, whose expectation is bounded by one under the null at every stopping time \citep{shafer2021testing}. Every non-negative supermartingale with initial value at most one is an e-process, but the converse need not hold in general.  This distinction is not merely formal: in some testing problems, e-processes can be strictly more powerful than test supermartingales. For example, when testing exchangeability, this difference is essential, since test supermartingales are powerless while suitable e-processes can  provide evidence against the null \citep{ramdas22testing}. Yet, test (super)martingales still often play a central role in practice. In the classical simple-null versus simple-alternative setting, the likelihood-ratio process is a test martingale under the null and underlies Wald's sequential probability ratio test. Moreover, for \emph{conditional} hypotheses (\emph{i.e.}, hypotheses obtained by extending a single-step testing problem to the sequential setting, imposing the same restrictions on the conditional distribution of each new observation given the past), test supermartingales in the form of products of one-step e-values are the most natural sequential objects, much as likelihood-ratio products arise in simple sequential testing. This raises the question of whether, in such cases, supermartingales are sufficient to dominate arbitrary e-processes. The present work gives a positive answer for a broad class of non-parametric conditional hypotheses.
\paragraph{Main contribution.} In this paper we consider non-parametric hypotheses defined by a finite-dimensional constraint map, whose conditional mean under the null is required to belong (for each step) to an arbitrary non-empty convex set. Our main result shows that every e-process for such a conditional constrained hypothesis is pathwise dominated by a predictable product of affine one-step e-variables. Hence, for this class of conditional non-parametric hypotheses, allowing arbitrary e-processes gives no additional power over restricting to test supermartingales. This establishes a \emph{complete-class theorem}: predictable products of affine one-step e-variables are not merely a convenient way to construct valid sequential tests, but are sufficient for seeking optimal procedures among arbitrary e-processes. Indeed, pointwise domination ensures that, for any rejection threshold, the dominating  supermartingale rejects no later than the original e-process, on every observation path.

We further show that a minimal complete class exists and consists of predictable products of maximal affine one-step e-variables.  Every e-process is therefore majorised by one of these products, and no product in this class admits a strict pointwise improvement by another e-process.  Under equality constraints, every affine factor is already maximal, while for more general convex constraints the minimal complete class may be strictly smaller.

The constrained framework that we consider covers several familiar non-parametric restrictions. Conditional mean constraints are the basic example, with bounded-mean betting and confidence-sequence methods recently studied by \citet{orabona2024tight}, \cite{waudby2024estimating}, \cite{shekhar2023near}, and \cite{chugg2025time}, among others. Joint restrictions on conditional means and second moments are also encompassed by our framework, and related e-process tests for mean and variance constraints were studied by \citet{fan2025testing}. Bounded higher-moment assumptions as well fit the framework and are common in heavy-tailed sequential learning and mean estimation \citep{agrawal2021regret,wang2023catoni}. Moreover, finite distributional constraints can often be encoded by indicator functions (for example, quantile restrictions are one of the examples discussed by \citealt{larsson2026testing}). Finally, the framework also covers the conditional extensions of arbitrary one-step hypotheses on finite sample spaces: convexifying the one-step null leaves the class of e-processes unchanged and brings the hypothesis within our setting. We will return to this observation in the final discussion at the end of this paper.

From a technical perspective, our proof builds on the convex-duality approach underlying the single-step results of \citet{clerico2024optimal} and \citet{larsson2026testing}. Extending this approach to the sequential setting is, however, technically non-trivial: it requires constructing a suitable continuation envelope and establishing its upper semi-analyticity. This provides the regularity needed to apply the Borel-selection theorem of \citet{clerico2026borel}, yielding a predictable choice of affine dominators.

\paragraph{Related work.} Test (super)martingales have long played a central role in sequential hypothesis testing. In the classical simple-null versus simple-alternative setting, the likelihood-ratio process is a test martingale under the null and underlies Wald's sequential probability ratio test \citep{wald1945sequential}. More generally, when the null is simple, every e-process is pointwise dominated by a non-negative martingale. For composite nulls, \citet{ramdas2020admissible} show that non-negative martingales remain universal for anytime-valid inference in an admissibility sense: after fixing a single distribution in the null, every e-process can be dominated by a martingale. This pointwise martingale domination, however, does not by itself provide a single  dominating test supermartingale that is valid uniformly over the whole composite null. 

In general non-parametric settings, a first positive rigorous result in the direction of the present work was obtained by \citet{clerico2025optimality} for bounded one-dimensional conditional mean sequential testing: every e-process for the conditional mean hypothesis is dominated by a \emph{coin-betting} martingale. The proof in \citet{clerico2025optimality} is based on two main ingredients. First, the single-step mean constraint yields an explicit affine characterisation of the relevant one-step e-variables. Second, this single-step description can be iterated over time by choosing the affine coefficient predictably, as a function of the past. The first ingredient has since been extended beyond bounded one-dimensional mean testing. In \citet{clerico2024optimal}, the single-step affine-domination argument is developed for a broader class of non-parametric hypotheses defined through finitely many sufficiently regular linear constraints. Subsequently, \citet{larsson2026testing} generalised this result further, allowing essentially no regularity assumptions and an arbitrary number of linear constraints. These results, however, apply to the single-step setting and do not tackle the conditionally-extended sequential case. Whether or not test supermartingales (in the form of products of predictably chosen one-step affine e-values) are sufficient to dominate all e-processes for the corresponding conditional sequential hypotheses was highlighted as a relevant open question by both works.

Complete-class theorems have a long tradition in statistical decision theory \citep{lehmann2005testing,berger1985statistical}, including classical sequential hypothesis testing \citep{brown1980complete}. While classical results are typically formulated in terms of risk domination, the complete-class property considered here  has a different sense: domination holds on every observation path, without specifying a loss function, and it is formulated within the class of e-processes. The generality of e-processes is particularly relevant here. As discussed by \citet{ramdas2026universality}, every anytime-valid $\alpha$-level sequential test can be represented by thresholding an e-process. Thus, establishing pointwise domination of arbitrary e-processes by test supermartingales has a direct statistical consequence: any such sequential test can be replaced by one based on a test supermartingale that rejects no later, on every observation path.

\paragraph{Notation.}
Given two integers $a$ and $b$, we write $[a:b]$ for the set of integers between $a$ and $b$, endpoints included, which is empty if $a>b$. This notation also applies to $[a:\infty]=\{a,a+1,\dots\}\cup\{\infty\}$. 

Let $\Z$ be any set. For an integer $k\geq 1$, we write $\Z^k$ for the $k$-fold Cartesian product of $\Z$. Elements of $\Z^k$ are denoted by $z^k=(z_1,\dots,z_k)$. We also set $\Z^0=\{\emptyset\}$ and write $z^0=\emptyset$. We denote the concatenation of two vectors $x^t$ and $y^s$ by $(x^t,y^s)$. The space of infinite sequences with values in $\Z$ is denoted by $\Z^\infty$. Its elements are typically written either as $(z_t)_{t\geq t_0}$, when the starting index $t_0$ is relevant, or simply as $z^\infty$. For $T\geq t_0$, we write $z^{T:\infty}$ for the tail sequence starting from time $T$. For finite indices $a\leq b$, we write $z^{a:b}=(z_a,\dots,z_b)$. 

If $\Z$ is a measurable space, $P$ is a probability measure on $\Z$, and $f:\Z\to D$ is a measurable function, where $D$ is a measurable subset of $[-\infty,\infty]^m$, we write $\E_P[f(X)]$, or simply $\E_P[f]$, for the expectation of $f$ under $P$, whenever this expectation is well defined. In expectations, integrals, and sums representing expectations, we use the usual measure-theoretic convention $0\cdot\infty=0$, unless stated otherwise. Subsets of $\R^n$ and $[-\infty,\infty]^n$ are endowed with their Borel sigma-fields, generated by the usual Euclidean topology and the usual extended-real topology, respectively. We write  $\delta_x$ for the Dirac probability measure at $x$ and $P\otimes Q$ for the product of probability measures $P$ and $Q$.

We use the standard notation $\conv A$, $\aff A$, and $\relint A$ for the convex hull, affine hull, and relative interior of a set $A$, respectively. The corresponding definitions and convex-analytic facts used in the paper are collected in \Cref{app:convex}. For finite-dimensional vectors, $x\cdot y$ denotes the Euclidean inner product, while $\|x\|_1$ and $\|x\|_2$ denote the $\ell^1$ and Euclidean norms, respectively. When the subscript is omitted, $\|x\|$ denotes the Euclidean norm.
\section{Setting}
Throughout the paper, we let $\X$ be a standard Borel space. This assumption is satisfied by virtually all measurable spaces encountered in practice, including $\mathbb{R}^n$, any of its Borel subsets, and any Polish space equipped with its Borel sigma-field. We denote by $\PR$ the set of all Borel probability measures on $\X$. A \emph{hypothesis} is a subset $\Hh\subseteq \PR$, and an \emph{e-variable} for $\Hh$ is a Borel function $E:\X\to[0,\infty]$ such that $\E_P[E]\leq 1$ for every $P\in\Hh$.

This paper main focus is on the sequential testing framework, where multiple observations are observed. For any $T\geq 1$ (possibly $T=\infty$), we endow $\X^T$ with the product sigma-field (or the sigma-field generated by cylinders if $T=\infty$). With this choice, $\PR^T$  is the set of probability measures on $\X^T$. We let $\T_T$ denote the set of all stopping times\footnote{We always consider stopping times  relative to the natural filtration, induced by the projections on the components.} valued in $[0:T]$ ($\T_\infty$ is the set of all stopping times, possibly non-finite), while $\T_\star=\bigcup_{T=1}^\infty\T_T$ is the set of all bounded stopping times. A \emph{non-negative process} on $\X^\infty$ is a sequence $(f_t)_{t\geq 0}$ of Borel functions, with $f_t:\X^t\to[0,\infty]$ for all $t$. Given a hypothesis $\Hh\subseteq\PR^\infty$, an e-process $E = (E_t)_{t\geq 0}$ is a non-negative process that satisfies $E_0\leq 1$ and  $\E_P[E_\tau]\leq 1$, for every $\tau\in\T_\star$ and $P\in\Hh$. Notably, it is equivalent to require the same property  for every $\tau\in\T_\infty$ (see Lemma 6 in \citealt{ramdas2020admissible}).

We say that an e-variable $E'$ \emph{majorises} an e-variable $E$ if $E'(x)\geq E(x)$ for every $x\in\X$. Similarly, an e-process $E'$ majorises an e-process $E$ if, for every $x^\infty\in\X^\infty$ and every $t\geq 0$, $E'_t(x^t)\geq E_t(x^t)$. These order relations are natural for testing: replacing an e-variable or e-process by a larger valid one can only make the resulting threshold tests more powerful.
\paragraph{Constrained hypotheses.}
Let $\Phi:\X\to\R^m$ be a Borel map. We denote by $\PR_\Phi$ the set of probability measures under which $\Phi$ is integrable:
$$\PR_\Phi = \{P\in\PR\,:\,\E_P[\|\Phi\|_1]<\infty\}\,.$$
Let $S\subseteq\R^m$ be a non-empty convex set. We define the single-step constrained hypothesis associated with $(\Phi,S)$ as
$$\HhP = \{P\in\PR_\Phi\,:\,\E_P[\Phi]\in S\}\,.$$

We now pass to the conditional sequential setting. For $T\in[1:\infty]$, define
$$\PR_\Phi^T = \{P\in\PR^T\,:\,\E_P[\|\Phi(X_t)\|_1]<\infty\,,\;\forall\text{ finite }t\in[1:T] \}\,.$$
The corresponding conditional constrained hypothesis on $\X^T$ is
$$\HhP^T = \{P\in\PR_\Phi^T\,:\,\E_P[\Phi(X_t)\mid X^{t-1}]\in S\,,\;\forall\text{ finite } t\in [1:T]\,,\;P\text{-a.s.}\}\,.$$
For $t=1$, the conditioning is on the trivial history $X^0=\emptyset$, so that
$$\E_P[\Phi(X_1)\mid X^0]=\E_P[\Phi(X_1)]\,.$$

This formulation covers several standard examples. If $\X=[a,b]\subseteq\R$, $\Phi(x)=x$, and $S=\{\mu\}$, then $\HhP$ is the class of distributions on $[a,b]$ with mean $\mu$, and $\HhP^T$ imposes $\E_P[X_t\mid X^{t-1}]=\mu$ at each time. This is the bounded-mean setting underlying betting-based confidence sequences \citep{orabona2024tight,waudby2024estimating}. Taking $\X=\R^d$, $\Phi(x)=x$, and $S=\{y\in\R^d\,:\,\|y\|_2\leq 1\}$ gives a mean constraint in the Euclidean unit ball. Finally, with $\X=\R$, $\Phi(x)=(x,x^2)$, and $S=\{\mu\}\times[0,B]$, where $B\geq\mu^2$, the single-step hypothesis becomes $\E_P[X]=\mu$ and $\E_P[X^2]\leq B$, and its conditional version imposes the analogous heavy-tail conditional constraints given the past \citep{agrawal2021regret,wang2023catoni}. 

We remark that, when $S$ is closed, the single-step hypothesis $\HhP$ can also be written in the language of \citet{larsson2026testing}: indeed, $\E_P[\Phi]\in S$ is equivalent to a family of linear expectation constraints in their constrained-hypothesis framework. The main focus of this paper, however, is the conditional sequential setting, and the single-step results below should be read mainly as a simplified guide to the objects and arguments that will later enter the sequential characterisation.

\section{Single-step e-variable characterisation}

We first study the single-step problem. As discussed above, this case is not the main object of the paper, but it introduces the affine e-variables that will later be selected predictably and iterated in the sequential setting.

We begin by identifying the \emph{effective domain}, namely the part of the sample space that is ``visible'' under the null. Define
$$\XP = \left\{x\in\X\,:\,\exists P\in\HhP\,,\;P(\{x\})>0\right\}\,.$$
Thus, $\XP$ is the set of points that have non-zero mass under at least one distribution in the null hypothesis. Points outside $\XP$ play no role for validity, and a dominating e-variable may therefore take value $+\infty$ on them.

\begin{lemma}\label{lemma:XP}
    $\XP$ is a Borel subset of $\X$ and $P(\XP)=1$ for every $P\in\HhP$. 
\end{lemma}
\begin{proof}
    See \Cref{app:tech}.
\end{proof}

We will consider e-variables that are affine functions of $\Phi$ on the effective domain $\XP$. We first introduce the support function $\sP:\R^m\to[-\infty,\infty]$ given by
$$\sP(\lambda) = \sup_{y\in S\cap\aff\Phi(\XP)}\lambda\cdot y\,,$$
with $\aff\emptyset = \emptyset$ and $\sup\emptyset = -\infty$ if necessary. The role of $\sP$ is to give the largest possible value of $\lambda\cdot y$ over the feasible values of the mean of $\Phi$. The intersection with $\aff\Phi(\XP)$ reflects the fact that, under the null, $\E_P[\Phi]$ can only lie in the affine hull generated by the values of $\Phi$ on $\XP$. Thus, $\sP(\lambda)$ is the relevant upper bound for $\lambda\cdot\E_P[\Phi]$, when $P\in\HhP$. We now define
$$\LP = \{\lambda\in\R^m\,:\,1+\lambda\cdot\Phi(x)\geq\sP(\lambda)\,,\;\forall x\in\XP\}\,.$$
Moreover, $\lambda\in\LP$ precisely when $1+\lambda\cdot\Phi(x)-\sP(\lambda)\geq 0$ for every point $x$ that can occur under the null. This ensures that, for each $\lambda\in\LP$, we may define the non-negative map $\e_\lambda:\X\to[0,\infty]$ by
$$\e_\lambda(x) = \begin{cases} 1+\lambda\cdot\Phi(x)-\sP(\lambda) & \text{if $x\in\XP$;}\\ +\infty & \text{if $x\notin\XP$.}\end{cases}$$
We denote by $\EeP$ the set $\EeP = \{\e_\lambda\,:\,\lambda\in\LP\}$.
\begin{lemma}\label{lemma:EePevar}
    The elements of $\EeP$ are e-variables for $\HhP$. 
\end{lemma}
\begin{proof}
    Fix $\lambda\in\LP$. $\e_\lambda$ is Borel-measurable by \Cref{lemma:XP}. Moreover, $\e_\lambda$ is non-negative by definition of $\LP$. Now fix $P\in\HhP$. By \Cref{lemma:XP}, $\XP$ is Borel and $P(\XP)=1$. Hence we may regard $P$ as a probability measure on $\XP$ and restrict $\Phi$ to $\XP$. Therefore, $\E_P[\Phi]\in\conv\Phi(\XP)\subseteq\aff\Phi(\XP)$ (this is \Cref{lemma:convexhull} in \Cref{app:convex} applied to $\Phi|_{\XP}$). Since $P\in\HhP$, we also have $\E_P[\Phi]\in S$. Therefore, $\E_P[\Phi]\in S\cap\aff\Phi(\XP)$. In particular, by the definition of $\sP$, we have $\lambda\cdot\E_P[\Phi]\leq\sP(\lambda)$. Consequently, $\E_P[\e_\lambda]\leq 1$.
\end{proof}

The set of all the e-variables for $\HhP$ can be  characterised in terms of  $\EeP$.
\begin{proposition}\label{prop:evar}
    A  Borel function $E:\X\to[0,\infty]$ is an e-variable for $\HhP$ if and only if there exists $\hat E\in\EeP$ such that $E(x)\leq\hat  E(x)$ for every $x\in\X$.
\end{proposition}
We do not give a full proof of \Cref{prop:evar} here, since it follows from the main sequential result, \Cref{thm:eprocess_maj}. Indeed, any e-variable $E$ for $\HhP$ can be embedded into the sequential setting by defining $\tilde E_0=1$ and $\tilde E_t(x^t)=E(x_1)$ for $t\geq 1$. Then $\tilde E$ is an e-process for $\HhP^\infty$, and applying \Cref{thm:eprocess_maj} with constant stopping time $\tau=1$ yields the desired domination. More explicitly, the relevant argument is the one used in the proof of \Cref{lemma:dom}: a separation argument gives an affine one-step e-variable that dominates the continuation value, and taking the continuation stopping time to be $\tau=0$ gives the present single-step statement. The following informal argument illustrates the separation idea behind this result in the special case of a mean-zero constraint for distributions on the unit ball. The proof for the general case follows similar principles (see proof of \Cref{lemma:dom} in \Cref{app:proof_dom}).
\begin{figure}[t]
\centering
\includegraphics[width=.85\textwidth]{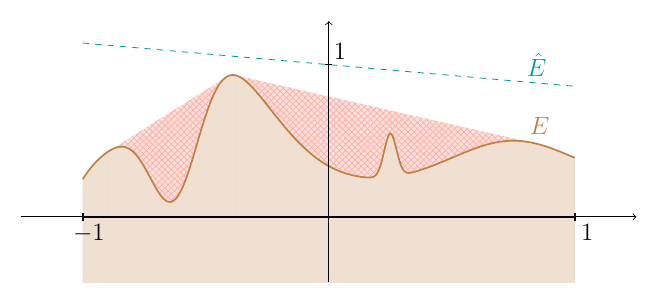}
\caption{Geometric intuition for the proof sketch. The brown solid curve represents an e-variable $E$, and the brown shaded region is its hypograph. Validity under all mean-zero distributions prevents the convex hull of this hypograph (brown region and pink hatched  region) from intersecting the vertical axis above $1$. A separating hyperplane then yields an affine e-variable $\hat E\in\EeP$, shown as the dashed line, that dominates $E$.}
\label{fig:proof}
\end{figure}
\begin{proof}[Proof sketch in a simplified setting]
    Let us consider the particular case in which $\X =B$ (the closed unit ball in $\R^m$), $\Phi$ is the identity, and $S=\{0\}$. Then the null $\HhP$ consists of all distributions on $\X$ with mean zero, and $\XP=\X$. In this case the affine e-variables take the form $\e_\lambda:x\mapsto 1+\lambda\cdot x$, with $\lambda\in\LP=B$.
    
    Let $E$ be an e-variable for $\HhP$. The condition $\E_P[E]\leq1$ for every null distribution says, in particular, that one cannot take finitely many points $x_1,\dots,x_n$, average them with non-negative weights $\alpha_1,\dots,\alpha_n$ (summing to one) so that $\sum_{i=1}^n \alpha_i x_i=0$, and at the same time have $\sum_{i=1}^n \alpha_i E(x_i)>1$. Indeed, such a choice would define a distribution in $\HhP$ under which the expectation of $E$ is larger than one.
    
    Geometrically, consider the points $(x,E(x))$ in $\R^m\times\R$, together with everything below them. The preceding observation says that the convex hull $C$ of this lower set cannot intersect the vertical half-line above the point $(0,1)$. Standard convex analytic results ensure that there must exist a hyperplane separating $C$ from $(0,1)$. (See \Cref{fig:proof} for a one-dimensional pictorial representation.) Algebraically, this means that there is a vector $\lambda\in\R^m$ such that $E(x)\leq \e_\lambda(x) = 1+\lambda\cdot x$ for all $x\in B$. Since $E$ is non-negative, this also means that $\e_\lambda\geq0$ for all $x\in B$, which implies $\lambda\in B$. For any $P\in\HhP$, $\E_P[\e_\lambda] = 1 + \lambda\cdot\E_P[X]=1$, so the right-hand side is itself an e-variable. Thus, in this simple case, every e-variable is majorised by an affine one.

    On the other hand, if a non-negative Borel function $E:B\to[0,\infty]$ is majorised by $\e_\lambda$, with $\lambda\in B$, then $\E_P[E]\leq\E_P[\e_\lambda]\leq 1$, so $E$ is an e-variable. We conclude that a non-negative Borel function is an e-variable if and only if it is dominated by an element of $\EeP$.
\end{proof}

\section{E-process characterisation}

We now pass from the single-step setting to the sequential one, where we want to characterise the set of e-processes. We will see that every e-process is dominated by a supermartingale in the form of a product of affine e-variables, where at each time the affine coefficient can be chosen based on the  past observations. This gives the natural sequential analogue of the class $\EeP$.

Let $\LP^\infty$ denote the set of sequences $\lambda^\infty=(\lambda_t)_{t\geq1}$ such that, for each $t\geq1$, the map
$$\lambda_t:\XP^{t-1}\to\LP$$
is Borel. For $\lambda^\infty\in\LP^\infty$, define the non-negative process $E^{\lambda^\infty}=(E_t^{\lambda^\infty})_{t\geq0}$ recursively by setting $E_0^{\lambda^\infty}=1$ and, for $t\geq1$,
$$E_t^{\lambda^\infty}(x^t)=\begin{cases}E_{t-1}^{\lambda^\infty}(x^{t-1})\e_{\lambda_t(x^{t-1})}(x_t) & \text{if $x^{t-1}\in\XP^{t-1}$;}\\+\infty & \text{otherwise.}\end{cases}$$
In the above product, we adopt the \emph{non-standard} convention $0\cdot\infty=\infty$. This  is irrelevant under the null, as paths remain in $\XP$ almost surely, but it ensures that domination statements also hold outside of the effective domain.

We define the set  $\EeP^\infty=\{E^{\lambda^\infty}:\lambda^\infty\in\LP^\infty\}$, the class of predictable products of affine one-step e-variables. 

\begin{lemma}\label{lemma:EePinfeproc}
    The elements of $\EeP^\infty$ are e-processes for $\HhP^\infty$.
\end{lemma}
\begin{proof}
    Fix $\lambda^\infty\in\LP^\infty$ and let $E^{\lambda^\infty}$ be the corresponding process. The process is non-negative by construction. Moreover, $\sP$ is lower semi-continuous as a supremum of affine functions, and in particular it is Borel. This, with  \Cref{lemma:XP} and the Borel measurability of the maps $\lambda_t$, ensures that the process $E^{\lambda^\infty}$ is Borel.

    Let $P\in\HhP^\infty$. Since $P(\XP^\infty)=1$, the values assigned outside $\XP^\infty$ are irrelevant under $P$. Fix $t\geq1$. Conditionally on $X^{t-1}$, the law of $X_t$ belongs to $\HhP$ for $P$-almost every history. By \Cref{lemma:EePevar} $\E_P\left[\e_{\lambda_t(X^{t-1})}(X_t)\mid X^{t-1}\right]\leq1$ $P$-almost surely, and so $\E_P[E_t^{\lambda^\infty}\mid X^{t-1}]\leq E_{t-1}^{\lambda^\infty}$, $P$-almost surely. We conclude that $E^{\lambda^\infty}$ is a non-negative supermartingale with initial value one. By Doob's optional stopping theorem for bounded stopping times (see, e.g., Theorem 10.4.1 in \citealt{dudley2002real}), for every $\tau\in\T_\star$ we have $\E_P[E_\tau^{\lambda^\infty}]\leq1$. Since $P\in\HhP^\infty$ was arbitrary, $E^{\lambda^\infty}$ is an e-process for $\HhP^\infty$.
\end{proof}

The main result of this paper is the following domination statement: every e-process for the conditional constrained hypothesis is dominated, path by path and time by time, by one of these predictable affine products. This yields a full characterisation of the e-processes for $\HhP^\infty$. 

\begin{theorem}\label{thm:eprocess_maj}
    A non-negative process $E = (E_t)_{t\geq 0}$ is an e-process for $\HhP^\infty$ if and only if there is $\hat E = (\hat E_t)_{t\geq 0}\in\EeP^\infty$, such that $E_t(x^t)\leq \hat E_t(x^t)$ for every $t\geq 0$ and  $x^t\in\X^t$. 
\end{theorem}
\begin{proof}
    The ``if'' direction is immediate. If $E$ is majorised by some $\hat E\in\EeP^\infty$, then for every $P\in\HhP^\infty$ and every $\tau\in\T_\star$, $\E_P[E_\tau]\leq \E_P[\hat E_\tau]\leq 1$ by \Cref{lemma:EePinfeproc}. Hence $E$ is an e-process.

    To prove the converse, we first introduce the following notation and terminology. Let $E$ be an e-process for $\HhP^\infty$, for $t\geq 1$ and $x^t\in\X^t$, we define the process $E^{x^t} = (E_s^{x^t})_{s\geq 0}$ by setting $E_0^{x^t}=E_t(x^t)$ and $E_s^{x^t}(y^s) = E_{t+s}(x^t, y^s)$ for $s\geq 1$ and $y^s\in\X^s$. For $T\geq 1$, we say that $\lambda^T\in\LP^T$ dominates $E$ at level $T$ if, for all $\tau\in\T_\star$, all $Q\in\HhP^\infty$, all finite $t\in[1:T]$ and all $x^t\in\X^t$, 
    \begin{equation}\label{eq:defdom}\E_Q[E_\tau^{x^t}]\leq E^{\lambda^T}_t(x^t)\,.\end{equation} where $E^{\lambda^T}(x^t) = \prod_{s=1}^t \e_{\lambda_s(x^{s-1})}(x_s)$,  with the convention $0\cdot\infty = \infty$.

    If $\HhP$ is empty, then $\XP=\emptyset$ and the result is trivial. So, we will henceforth assume that $\HhP\neq\emptyset$. First, we prove the claim under the additional assumptions that $\X = \XP$ and $S$ is Borel (note that, in general, a convex set in $\R^m$ need not be Borel). Let $E$ be an e-process for $\HhP^\infty$.  We will construct inductively a sequence $\lambda^\infty = (\lambda_t)_{t\geq 1}\in\LP^\infty$ such that, for every $T\geq 1$, the truncation $\lambda^T = (\lambda_1,\dots, \lambda_T)$ dominates $E$ at level $T$. 

    The initial step of the induction is the following one-step domination lemma. We defer to \Cref{app:proof_dom} the proof, whose  argument is the rigorous version of the separation idea illustrated in the proof sketch of \Cref{prop:evar} in the previous section.
    \begin{lemma}\label{lemma:dom}
        Assume $\HhP\neq\emptyset$ and $\X=\XP$. Let $E$ be an e-process for $\HhP^\infty$. Then there exists $\lambda\in\LP$ that dominates $E$ at level $1$.
    \end{lemma}
    
    Now, fix $T\geq2$, and assume that $\lambda^{T-1}=(\lambda_1,\dots,\lambda_{T-1})$ dominates $E$ at level $T-1$. We will construct a Borel map $\lambda_T:\X^{T-1}\to\LP$ such that $\lambda^T=(\lambda^{T-1},\lambda_T)$ dominates $E$ at level $T$.

    Let $$B_{T-1} = \{x^{T-1}\in\X^{T-1}\,:\,E_{T-1}^{\lambda^{T-1}}(x^{T-1})>0\}\,.$$
    Since $E^{\lambda^{T-1}}_{T-1}$ is Borel, $B_{T-1}$ is a Borel subset of $\X^{T-1}$. Fix $\bar x^{T-1}\in B_{T-1}$. We define a non-negative process $\tilde E^{\bar x^{T-1}} = (\tilde E_t^{\bar x^{T-1}})_{t\geq 0}$ by setting $$\tilde E^{\bar x^{T-1}}_t = \frac{E^{\bar x^{T-1}}_t}{E^{\lambda^{T-1}}_{T-1}(\bar x^{T-1})}$$
    for every $t\geq 0$. Since $\lambda^{T-1}$ dominates $E$ at level $T-1$, for every $Q\in\HhP^\infty$ and every $\tau\in\T_\star$, we have 
    $$\E_Q[\tilde E_\tau^{\bar x^{T-1}}] = \frac{\E_Q[E_\tau^{\bar x^{T-1}}]}{E_{T-1}^{\lambda^{T-1}}(\bar x^{T-1})}\leq 1$$ by \eqref{eq:defdom}.
    Thus, $\tilde E^{\bar x^{T-1}}$ is an e-process for $\HhP^\infty$. 
    
    Define the mapping $u:B_{T-1}\times\X\to[0,\infty]$ as
    $$u(x^{T-1},x) = \sup_{Q\in\HhP^\infty}\sup_{\tau\in\T_\star}\E_Q\big[(\tilde E^{x^{T-1}})_\tau^{x}]\,.$$
    For any $\bar x^{T-1}\in B_{T-1}$, since $\tilde E^{\bar x^{T-1}}$ is an e-process for $\HhP^\infty$, \Cref{lemma:dom} yields some $l_{\bar x^{T-1}}\in\LP$ such that \begin{equation}\label{eq:boundu}u(\bar x^{T-1}, x) \leq 1 + l_{\bar x^{T-1}}\cdot\Phi(x) - \sP(l_{\bar x^{T-1}})\,,\end{equation} for every $x\in\X$. In particular, $u$ is finite valued. 

    One might be tempted to define $\lambda_T:x^{T-1}\mapsto l_{x^{T-1}}$ on $B_{T-1}$ (and $0$ elsewhere) and go on with proving that $(\lambda^{T-1},\lambda_T)$ dominates $E$ at level $T$. However, things are more delicate. We have just shown that for every $x^{T-1}\in B_{T-1}$ there is some value $l_{x^{T-1}}$, but this does not imply any regularity of the mapping $x^{T-1}\mapsto l_{x^{T-1}}$, which in general is not Borel. So, we cannot use $x^{T-1}\mapsto l_{x^{T-1}}$ to define $\lambda_T$ directly. Yet, we can use the following  Borel selection result from \cite{clerico2026borel} to find a Borel map that satisfies the same inequality as in \eqref{eq:boundu}. 
    \begin{lemma}[Corollary 3 in \citealt{clerico2026borel}]\label{lemma:gem}
        Let $\A$ and $\B$ be standard Borel spaces. Let $f:\A\times\B\to\R$ be upper semi-analytic and $\phi:\B\to\R^m$ Borel. Assume that there exists a map $\gamma:\A\to\R^m$ such that for all $a\in\A$ and $b\in\B$, $f(a,b)\leq \gamma(a)\cdot\phi(b)$. Then, there exists a Borel map $\Gamma:\A\to\R^m$, satisfying $f(a,b)\leq \Gamma(a)\cdot\phi(b)$, for all $a\in\A$ and $b\in\B$. 
    \end{lemma}
    To show how the above lemma applies to our case, let $\A = B_{T-1}$ and $\B = \X\times (S\cap \aff\Phi(\X))$. Both $\A$ and $\B$ are standard Borel spaces (where we use that we are assuming that $S$ is Borel). Define $\tilde u:\A\times\B\to\R$ as $\tilde u(x^{T-1}, (x,z)) = u(x^{T-1},x)-1$ and $\phi:\B\to\R^m$ as $\phi(x,z) = \Phi(x)-z$. Let $\gamma:\A\to\R^m$ be given by $\gamma(x^{T-1}) = l_{x^{T-1}}$. \eqref{eq:boundu} now reads
    $$\tilde u(x^{T-1}, (x,z))\leq \gamma(x^{T-1})\cdot\phi(x,z)\,,$$ for all $x^{T-1}\in B_{T-1}$, $x\in\X$, and $z\in S\cap\aff\Phi(\X)$, where we used that $\sP(\lambda) = \sup_{z\in S\cap\aff\Phi(\X)}\lambda\cdot z$ for every $\lambda\in\R^m$, since $\X = \XP$. $\phi$ is a Borel map, so in order to use \Cref{lemma:gem} we only need to ensure that $\tilde u$ is ``regular enough'' (specifically upper semi-analytic, see below after the end of this proof, and \Cref{app:coarse-semian}, for a definition). This is a non-trivial technical point, it follows from \Cref{prop:continuation-envelope-usa} in \Cref{app:coarse}  that this is the case. So, \Cref{lemma:gem} applies and implies that there exists a Borel map $\lambda_T:B_{T-1}\to\R^m$ such that
    $$u(x^{T-1}, x)\leq 1 + \lambda_T(x^{T-1})\cdot\Phi(x) - \sP(\lambda_T(x^{T-1}))$$
    for every $x^{T-1}\in B_{T-1}$ and $x\in\X$. Since $u$ is non-negative and $\X=\XP$, we conclude that $\lambda_T(x^{T-1})\in\LP$ for every $x^{T-1}\in B_{T-1}$. We can also extend $\lambda_T$ to the whole $\X^{T-1}$ by setting $\lambda_T(x^{T-1}) = 0$ outside of $B_{T-1}$. Since $0\in\LP$, this ensures that we have a Borel map $\lambda_T:\X^{T-1}\to\LP$. 

    We now show that $\lambda^T = (\lambda^{T-1},\lambda_T)$ dominates $E$ at level $T$. First, fix $\bar x^T = (\bar x^{T-1},\bar x_T)\in B_{T-1}\times\X$, $Q\in\HhP^\infty$, and $\tau\in\T_\star$. Then, 
    $$\E_Q[E^{\bar x^T}_\tau] \leq E_{T-1}^{\lambda^{T-1}}(\bar x^{T-1})\, u(\bar x^{T-1},\bar x_T)\leq E_T^{\lambda^T}(\bar x^T)\,.$$
    On the other hand, fix $\bar x^{T} = (\bar x^{T-1}, \bar x_T) \in\X^T$, and assume that $\bar x^{T-1}\notin B_{T-1}$. Then, $E^{\lambda^{T-1}}_{T-1}(\bar x^{T-1})=0$ by definition of $B_{T-1}$. Fix  $Q\in\HhP^\infty$, and $\tau\in\T_\star$. Since $\X=\XP$,  there is $P\in\HhP$ with $P(\{\bar x_T\})=\varepsilon>0$. Let $Q_\star = P\otimes Q$. Then, $Q_\star\in\HhP^\infty$. Define the stopping time $\tau_\star$ by $\tau_\star(x^\infty) = 1 + \tau(x^{2:\infty})$ if $x_1 = \bar x_T$ and $1$ otherwise. Because $\lambda^{T-1}$ dominates $E$ at level $T-1$, we have
    $$0 = E_{T-1}^{\lambda^{T-1}}(\bar x^{T-1}) \geq \E_{Q_\star}[E^{\bar x^{T-1}}_{\tau_\star}] \geq \varepsilon \E_Q[E^{\bar x^T}_\tau]\,.$$
    So, $\E_Q[E^{\bar x^T}_\tau] = 0$, and as $E_T^{\lambda^T}(\bar x^T) = 0$ we still have $\E_Q[E_\tau^{\bar x^T}]\leq E_T^{\lambda^T}(\bar x^T)$. Since for any $t<T$, $E^{\lambda^T}_t = E^{\lambda^{T-1}}_t$, we can therefore conclude that $\lambda^T$ dominates $E$ at level $T$. 

    So far, we have shown that we can construct a sequence $\lambda^\infty\in\LP^\infty$, such that, for every $T\geq 1$, the truncation $\lambda^T$ dominates $E$ at level $T$. Now,  fix $T\geq 1$ and $x^T\in\X^T$. Taking $\tau = 0$ in the definition of domination at level $T$, we get
    $$E_T(x^T) \leq E_T^{\lambda^T}(x^T) = E_T^{\lambda^\infty}(x^T)\,.$$
    Since $T\geq1$ and $x^T$ are arbitrary (and we also have $E_0 \leq 1 = E^{\lambda^\infty}_0$), $E^{\lambda^\infty}$ majorises $E$. This concludes the proof under the additional assumptions that  $\X = \XP$ and $S$ is Borel.

    We now discuss the general case. Let $A=\aff\Phi(\XP)$ and $S'=\relint(S\cap A)$. $S'$ is convex and Borel as a relatively open convex subset of the affine space $A$. Moreover, \Cref{lemma:ropen} in \Cref{app:tech} ensures that $S'\neq\emptyset$, $\Hh_{\Phi,S'}\neq\emptyset$, $\XP = \X_{\Phi,S'}$, $\sP = \sigma_{\Phi,S'}$, and $\LP = \Lambda_{\Phi,S'}$. Let $E$ be an e-process for $\HhP^\infty$. Since $\Hh_{\Phi,S'}^\infty\subseteq\HhP^\infty$, the process $E$ is also an e-process for $\Hh_{\Phi,S'}^\infty$. Restricting to the space $\XP$, what established so far yields a $\lambda^\infty\in\LP^\infty$ such that the corresponding product process $E^{\lambda^\infty}$ dominates $E$ on $\XP^\infty$. It remains to check that $E^{\lambda^\infty}$ majorises $E$ on all of $\X^\infty$. If a path $x^\infty$ ever leaves $\XP$, then the first one-step factor evaluated outside $\XP$ is $+\infty$. With the convention $0\times\infty=\infty$ that we adopt, it follows that $E_t^{\lambda^\infty}(x^t)=+\infty$ from that time onward. Hence the majorisation holds also outside $\XP^\infty$. 
    \end{proof}

    \paragraph{Comments on the omitted measure-theoretic details.} As we mentioned in the proof, a crucial  delicate technical step is to find a Borel map $\lambda_T$, once it is known that for every history $x^{T-1}$ there exists some value $l_{x^{T-1}}$. The key tool that we use is Corollary 3 of \cite{clerico2026borel}. To apply it in our setting, we need the continuation envelope
    $$u(x^T)=\sup_{Q\in\HhP^\infty}\sup_{\tau\in\T_\star}\E_Q\big[(\tilde E^{x^{T-1}})^{x_T}_\tau\big]$$
    to be upper semi-analytic. If $\A$ and $\B$ are standard Borel spaces, a function $f:\A\to[-\infty,\infty]$ is upper semi-analytic if there is a Borel function $g:\A\times\B\to[-\infty,\infty]$ such that, for every $a\in\A$,
    $$f(a)=\sup_{b\in\B}g(a,b)\,.$$
    Although $u$ is already written as a supremum, the parameter space $\HhP^\infty\times\T_\star$ is not convenient, particularly because the class of bounded stopping times does not come with a simple suitable standard Borel structure. We therefore take an alternative route to show that $u$ is actually upper semi-analytic. In \Cref{app:coarse}, we introduce a coarsened version $\bHhP^\infty$ of the sequential hypothesis $\HhP^\infty$, obtained by restricting attention to laws whose conditional support has cardinality at most $d=\min(m+1,|\XP|)$ at each step. \Cref{cor:supsup} in \Cref{app:coarse} shows that this restriction does not change the supremum defining $u$. This makes the measurability problem much simpler. Under the coarsened hypothesis, bounded stopping times can be represented by stopping masks on finite trees. For each fixed horizon there are only finitely many such masks, so the supremum over stopping times bounded by that horizon is a finite maximum. Taking the supremum over all bounded stopping times then amounts to a countable supremum over the horizon. Since the remaining tree parameters form a standard Borel space, the resulting envelope $u$ is upper semi-analytic. 

\section{Minimal complete class}\label{sec:minimal}

So far, we have established that every e-process for $\HhP^\infty$ is majorised by an element of $\EeP^\infty$. Thus, $\EeP^\infty$ forms a complete class of e-processes. However, this class need not be minimal, since some of its elements may themselves be strictly majorised by other e-processes. In general, a minimal complete class need not exist (see, e.g., Proposition 2 in \citealt{clerico2024optimal}). However, in the bounded one-dimensional conditional mean setting, \citet{clerico2025optimality} showed that the class of coin-betting e-processes is minimal complete. We show that, in our more general setting, a minimal complete class also exists and consists precisely of predictable products of maximal affine one-step e-variables.

We call an e-variable or e-process \emph{maximal} if it is not strictly majorised by another e-variable or e-process for the same hypothesis. Here, strict majorisation means pointwise majorisation with strict inequality at least at one point (and at one time, in the sequential case). Clearly, if a minimal complete class exists, it necessarily coincides with the class of all maximal elements. Let
$$\LPmax=\{\lambda\in\LP\,:\,\e_\lambda\text{ is maximal}\}\,.$$
By \Cref{prop:evar}, maximality among all e-variables is equivalent to maximality within the affine class $\EeP$. 

\begin{lemma}\label{lem:mcc}
Assume $\HhP\neq\emptyset$. $\LPmax$ is Borel, and there exists a Borel map $J:\LP\to\LPmax$ such that $\e_{J(\lambda)}\geq\e_\lambda$ for every $\lambda\in\LP$, with $J(\lambda)=\lambda$ whenever $\lambda\in\LPmax$. 
\end{lemma}

The proof is deferred to \Cref{app:proof_mcc}. Together with \Cref{prop:evar}, \Cref{lem:mcc} already establishes the existence of a minimal complete class for the single-step hypothesis $\HhP$. Indeed, every e-variable is majorised by a maximal one, and the class of all maximal e-variables is therefore minimal complete (cf.~Section 1.3 in \citealt{wald1950statistical}). We denote by $\Ee_{\LPmax}$ the class of affine e-variables indexed by $\LPmax$, and by $\Ee_{\LPmax}^\infty$ the corresponding class of predictable products, defined as $\EeP^\infty$ but with coefficients taking values in $\LPmax$.

\begin{proposition}\label{prop:mcc}
Assume $\HhP\neq\emptyset$. Then $\Ee_{\LPmax}$ and $\Ee_{\LPmax}^\infty$ are the minimal complete classes of e-variables for $\HhP$ and e-processes for $\HhP^\infty$, respectively. 
\end{proposition}

\begin{proof}
We have already discussed that for the single-step setting the claim follows directly from \Cref{prop:evar} and \Cref{lem:mcc}.

For the sequential setting, fix an arbitrary e-process $E$. By \Cref{thm:eprocess_maj}, there exists $\lambda^\infty\in\LP^\infty$ such that $E^{\lambda^\infty}$ majorises $E$. Applying the map $J$ from \Cref{lem:mcc} to each predictable coefficient yields a strategy taking values in $\LPmax$, whose corresponding product majorises $E^{\lambda^\infty}$. Thus, $\Ee_{\LPmax}^\infty$ is complete.

It remains to prove that its elements are maximal. Let $\lambda^\infty$ be a predictable strategy taking values in $\LPmax$, and suppose that an e-process $E'$ majorises $E^{\lambda^\infty}$. By \Cref{thm:eprocess_maj}, there exists $\eta^\infty\in\LP^\infty$ such that $E^{\eta^\infty}$ majorises $E'$. We show inductively that $E^{\eta^\infty}=E^{\lambda^\infty}$. At time zero both processes equal one. Suppose they coincide up to time $t-1$. Fix a history $x^{t-1}\in\XP^{t-1}$ and let
$w = E_{t-1}^{\eta^\infty}(x^{t-1})=E_{t-1}^{\lambda^\infty}(x^{t-1})$. If $w>0$, then majorisation at time $t$ implies that $\e_{\eta_t(x^{t-1})}$ majorises $\e_{\lambda_t(x^{t-1})}$. But $\lambda_t(x^{t-1})\in\LPmax$ and so $\e_{\lambda_t(x^{t-1})}$ is maximal. Thus, the two factors must coincide. On the other hand, when $w=0$, both $E^{\eta^\infty}_t$ and $E^{\lambda^\infty}_t$ equal zero on the effective domain and $+\infty$ outside the effective domain. Hence the processes also coincide at time $t$.

Therefore, $E'=E^{\lambda^\infty}$, so every element of $\Ee_{\LPmax}^\infty$ is maximal. 
\end{proof}

A particularly simple consequence concerns hypotheses defined by equality constraints.

\begin{corollary}\label{cor:mcc}
Assume $\HhP\neq\emptyset$ and $S=\{\mu\}$ for some $\mu\in\R^m$. Then $\LPmax=\LP$. In particular, the class $\EeP^\infty$ of \Cref{thm:eprocess_maj} is already minimal complete.
\end{corollary}

\begin{proof}
For every $\lambda\in\LP$ and every $P\in\HhP$, we have $\E_P[\e_\lambda]=1$. Since each point of $\XP$ has positive probability under at least one null distribution, no e-variable can strictly majorise $\e_\lambda$ without violating the e-variable condition. Thus every $\e_\lambda$ is maximal, and the conclusion follows from \Cref{prop:mcc}.
\end{proof}

\section{Examples}

We illustrate \Cref{thm:eprocess_maj} and \Cref{prop:mcc} in four conditional testing problems. In each example, $\XP=\X$, and we determine $\LP$ and $\LPmax$, which give the corresponding complete and minimal complete classes of predictable affine products. The examples include both settings in which the full affine class is already minimal complete and settings in which restricting to maximal factors yields a strictly smaller class.

\paragraph{Bounded multivariate mean.}
Let $\X=B^d=\{x\in\R^d:\|x\|_2\leq1\}$, $\Phi$ be the identity, and $S=\varepsilon B^d$, where $0\leq\varepsilon<1$. The conditional null requires the mean of each observation, given the past, to lie within distance $\varepsilon$ of the origin. Since $\sP(\lambda)=\varepsilon\|\lambda\|_2$, we have $\LP=(1+\varepsilon)^{-1}B^d$ and $\e_\lambda(x)=1+\lambda\cdot x-\varepsilon\|\lambda\|_2$. In this case $\LPmax=\LP$. Indeed, $\e_\eta\geq\e_\lambda$ implies
$$\|\eta-\lambda\|_2 = \sup_{x\in B^d}(\eta-\lambda)\cdot x\leq\varepsilon(\|\lambda\|_2-\|\eta\|_2)\leq\varepsilon\|\eta-\lambda\|_2\,,$$
which forces $\eta=\lambda$. The e-variable $e_\lambda$ can therefore be interpreted as a bet, where the direction of $\lambda$ specifies the anticipated deviation from the null, while $\|\lambda\|$ determines the size of the stake.

We remark that anytime-valid inference for the mean of distributions on $\R^d$ was considered by \citet{chugg2025time}, who construct non-negative supermartingales using products of exponential factors. In the bounded setting, these generally lie outside our minimal complete class and can be majorised by predictable affine products. However, their approach yields particularly tractable and explicit confidence spheres, and it is unclear whether similarly convenient results can be obtained using affine products.

\paragraph{Median.}
Let $\X=\R$ and fix $q\in\R$. We consider the null  stating that $q$ is a conditional median. This can be encoded by $\Phi(x)=(\one_{\{x<q\}},\one_{\{x>q\}})$ and $S=(-\infty,1/2]^2$. Writing $\lambda=(a,b)$, the support function is $\sP(a,b)=
(a+b)/2$ if $a,b\geq0$, and $+\infty$ otherwise. Therefore, 
$\LP=\{(a,b):a,b\geq0\,,\;a+b\leq2\}$ and $\e_{a,b}(x)=1+a(\one_{\{x<q\}}-1/2)+b(\one_{\{x>q\}}-1/2)$.
The  coefficients in $\LPmax$ are exactly those for which either $a$ or $b$ is zero. Indeed, let $c=\min(a,b)$. Then $(a-c,b-c)\in\LP$ and $\e_{a-c,b-c}(x)-\e_{a,b}(x)=c\,\one_{\{x=q\}}$.
Conversely, suppose $c=0$ and $\e_{a',b'}\geq\e_{a,b}$ for some $(a',b')\in\LP$. Evaluating at $x<q$ and $x>q$ gives $a'-b'\geq a-b$ and $a'-b'\leq a-b$, hence $a'-b'=a-b$. The case $x=q$ yields $a'+b'\leq a+b=|a-b|=|a'-b'|\leq a'+b'$,
which forces $(a',b')=(a,b)$. Hence the minimal complete class consists of predictable products that  bet only on one side of the proposed median at each step.

We note that this example provides a minimal complete class interpretation of the Bernoulli reductions used in sequential quantile inference of \citet{howard2022sequential}, whose one-sided Bernoulli likelihood-ratio factors corresponds to the minimal complete class for the conditional-median hypothesis.

\paragraph{Zero-mean under a covariance bound.}
Let $\X=\R^d$, $\Phi(x)=(x,xx^\top)$, and
$S=\{0\}\times\{M\succeq0:M\preceq \Id\}$, identifying symmetric matrices with a Euclidean space equipped with the trace inner product. Thus, the conditional null imposes a zero mean and a covariance matrix bounded above by $\Id$. Writing $\lambda=(a,A)\in\R^d\times\R^{d\times d}$, with $A$ symmetric, the support function is
$\sP(a,A)=\sum_{i=1}^d\max\{\lambda_i(A),0\}$,
where $\lambda_i(A)$ are the eigenvalues of $A$. Non-negativity of the affine factors on $\R^d$ yields
$$\LP=\left\{(a,A)\in\R^d\times\R^{d\times d}\,:\,A=A^\top\,,\;
\begin{pmatrix}
A & a/2\\
a^\top/2 & 1-\Tr(A)
\end{pmatrix}\succeq0\right\}\,,$$
with corresponding factors $\e_{a,A}(x)=1+a\cdot x+x^\top Ax-\Tr(A)$.
Again, we have that $\LPmax=\LP$. Indeed, every factor has expectation one under any distribution with mean zero and covariance $\Id$, and every $x\in\R^d$ can receive positive mass under some such distribution. Thus no e-variable in $\EeP$ admits a strict pointwise improvement. 

For $d=1$, our minimal complete class matches the quadratic factors used by \cite{fan2025testing} and \cite{agrawal2021regret}. Moreover, related multivariate mean inference under second-moment bounds was studied by \citet{chugg2025time}, using exponential-factor supermartingales. Under the present null, these can be majorised by predictable quadratic products, although obtaining similarly explicit confidence spheres from the latter may be more difficult.

\paragraph{Axial concentration on the sphere.}
As a last example, we consider sequential testing for excessive alignment of three-dimensional axial observations, where $x$ and $-x$ represent the same axis. Such alignment is naturally quantified through second-moment matrices \citep{kloppenburg2024functional}, which also play a role in tests of spherical uniformity \citep{cutting2020power}.

Let $\X=\mathbb S^2\subset\R^3$ and $\Phi(x)=xx^\top$, identifying symmetric matrices with a Euclidean space equipped with the trace inner product. For $1/3\leq\rho<1$, take
$$S=\{M\succeq0:\Tr(M)=1\,,\;M\preceq\rho I_3\}\,.$$
The conditional null requires the largest eigenvalue of the conditional second-moment matrix to be at most $\rho$, or equivalently bounds the expected squared projection along every axis.

Let $A$ be symmetric, with eigenvalues $a_1\geq a_2\geq a_3$. We distinguish two cases, whose details are deferred to \Cref{app:axial}.

For $1/2\leq\rho<1$, we have that $\sP(A)=\rho a_1+(1-\rho)a_2$ and $$\LP = \{A\,:\, \rho(a_1-a_2)+(a_2-a_3)\leq1\}\,.$$
The corresponding e-variables take the form $e_A(x) = 1 + x^\top A x -\sP(A)$. 
However, not all these e-variables are maximal. In fact, it turns out that $\LPmax=\{A\in\LP\,:\,a_2=a_3\}$. Moreover, noting that $e_A=e_{A+\xi I_3}$ for every $\xi\in\R$, we can eliminate this redundancy and parametrise the maximal affine e-variables as
$$\Ee_{\LPmax}=\{e_{\beta,u}\,:\,u\in\mathbb S^2\,,\;\beta\in[0,1/\rho]\}\,,$$
where $e_{\beta,u}(x)=1+\beta((u^\top x)^2-\rho)$.

For $1/3\leq\rho<1/2$, we instead have $\sP(A)=\rho(a_1+a_2)+(1-2\rho)a_3$ and
$$\LP=\{A\,:\,\rho(a_1+a_2-2a_3)\leq1\}\,.$$
Interestingly, in this case all admissible affine e-variables are maximal, so $\LPmax=\LP$. Again, using the invariance $e_A=e_{A+\xi I_3}$, we can eliminate the redundancy and parametrise the maximal affine e-variables as
$$\Ee_{\LPmax}=\{e_{\beta,\gamma, u, v}\,:\,u,v\in\mathbb S^2\,,\;u^\top v=0\,,\;\beta,\gamma\geq0\,,\;\beta+\gamma\leq1/\rho\}\,,$$
where $e_{\beta,\gamma, u, v}(x)=1+\beta((u^\top x)^2-\rho)+\gamma((v^\top x)^2-\rho)$.

\section{Discussion}
We have shown that, for conditionally finitely constrained hypotheses, every e-process is majorised by a supermartingale in the form of a predictable product of affine single-step e-variables. Since pointwise dominance corresponds to a more powerful  test, this is a complete-class theorem in the usual statistical sense (see, e.g., Chapter~1 in \citealt{lehmann2005testing}). Moreover, we have established that the minimal complete class exists and consists of predictable products of maximal affine single-step e-variables. In the single-step setting,  complete-class results for e-value testing of non-parametric hypotheses defined through linear constraints were obtained by \citet{clerico2024optimal} and \citet{larsson2026testing}. However, passing from the single-step result to the sequential one is not only a matter of iteration: even if an affine dominator exists after every fixed history, it must still be possible to choose such dominators predictably as functions of the past. \citet{larsson2026testing} explicitly raise as an open question the problem of extending their one-step characterisation to conditionally constrained hypotheses, pointing to \citet{clerico2025optimality} as the first positive result in such direction, although limited to bounded one-dimensional conditional mean testing. The present paper covers the broader case of conditional hypotheses defined via finitely many constraints. We remark that, differently from \cite{larsson2026testing}, we formulate domination pointwise, rather than only quasi-surely under the null. This makes no difference for validity,  but it does matter for testing: the realised path may come from an alternative, and in particular may leave the part of the space that is visible under the null. On the other hand, such a requirement might be too strong for parametric settings where domination is usually understood  up to null sets under a common reference measure.

The restriction to finitely many constraints is one main limitation of the present work. It excludes, for instance, natural hypotheses such as conditional sub-Gaussianity (cf.~\citealp{larsson2026testing}). Extending the domination result of \Cref{thm:eprocess_maj}  to infinitely many constraints remains open. More broadly, it would be interesting to understand whether this is an instance of a general principle: whenever a one-step complete class is available for a null hypothesis, do predictable products of elements of that class form a complete class for the corresponding conditional null? A positive answer would give a general route from one-step e-value characterisations to sequential e-process characterisations. At present, the main obstruction seems to be measurability: fixed-history domination is not enough, since the dominating one-step object must be selected predictably as a function of the past. The Borel-selection result of \citet{clerico2026borel} provides such a choice in the finite-dimensional affine setting, and explicitly leaves open the problem of extending this type of selection theorem beyond it. One possible route would be to weaken the measurability requirements, for instance by focusing on universal measurability.

In any case, the current finite-dimensional case covered by \Cref{thm:eprocess_maj} is  enough to characterise the e-processes for arbitrary conditional hypotheses on finite sample spaces, after convexification. Indeed, if $\X=\{1,\dots,n\}$ and $\mathcal H\subseteq\Delta_n$ (the probability simplex on $n$ points) is any single-step null, then the conditional null generated by $\mathcal H$ has the same e-processes as the conditional null generated by $\conv\mathcal H$: on each finite tree, kernels taking values in $\conv\mathcal H$ can be expanded node by node as finite mixtures of kernels taking values in $\mathcal H$. Now take $\Phi(i)=u_i\in\R^n$, where $u_i$ is the $i$-th coordinate vector. Then, for every law $P$ on $\X$, $\E_P[\Phi]=(P(\{1\}),\dots,P(\{n\}))$. Hence, taking $S=\conv\mathcal H$, the convexified null is exactly of the form $\Hh_{\Phi,S}$. Thus, arbitrary finite-state conditional hypotheses are covered by the present theorem. However, for the finite setting this result could also be proved more directly, avoiding the measurable-selection machinery used here.

A useful feature of our main result is that the complete class is explicitly parametrised by predictable choices of coefficients in the finite-dimensional set $\LP$. This kind of reduction can be valuable in practice. For example, when a simple alternative distribution $Q$ is given, one may look for a growth-rate-optimal e-variable, in the sense of \citet{grunwald2024safe}, namely an e-variable maximising $\E_Q[\log E(X)]$. By \Cref{prop:evar}, for our constrained hypotheses this optimisation can be restricted to $\EeP$. This becomes a finite-dimensional concave maximisation problem: indeed, $\lambda\mapsto \E_Q[\log\e_\lambda]$ is concave because $\sP$ is convex (as the supremum of affine functions), and it is easily checked that $\LP$ is closed and convex. On a similar note, \citet{arnold2026optimal} have used the complete-class result for coin-betting mean-testing to study GROW and REGROW optimality criteria for bounded mean testing within a tractable class of betting e-values, rather than over an arbitrary space of non-negative measurable functions. The completeness of the one-dimensional coin-betting class also provided a broader justification for the betting framework used by \citet{voracek2025star}, \citet{orabona2026optimal}, and \citet{waudby2024estimating}. The present work extends this perspective to more general conditional constrained hypotheses. For instance, our minimal complete class provides a natural search space for studying optimality under criteria such as those concerning growth and rejection time considered by \citet{waudby2026instance}. Finally, the time-sensitive framework of \citet{clerico2026timesensitive} formulates anytime-valid testing as a sequential control problem over a parametrised action space of one-step e-variables. Such a formulation requires the action space and the map from actions to bets to have enough measurable structure for  dynamic-programming arguments to hold. Our result provides precisely this kind of lossless parametrisation for conditional constrained nulls: instead of optimising over all e-processes, one may restrict  to predictable  products indexed by $\LP$.

Finally, an interesting remark concerns e-value merging. The sequential merging result of \citet{vovk2021merging} can be recovered from our framework by taking $\X=[0,\infty)$, $\Phi(x)=x$, and $S=[0,1]$. Indeed, the conditional null then requires $\E[X_t\mid X^{t-1}]\leq1$, and our theorem shows that every sequential e-merging function is dominated by a predictable product of affine betting factors.  Similarly, in the single-step setting, taking $\X=[0,\infty)^N$, $\Phi(x)=x$, and $S=[0,1]^N$ recovers the characterisation of e-merging functions under arbitrary dependence \citep{wang2024merging,clerico2026merging}. Other dependence structures, such as exchangeability, can also be expressed through constraints \citep{larsson2026testing}. Applying our results would likely bring explicit e-mergers characterisations in such settings. We leave a more systematic investigation of these connections for future work. %Notably, there are also related one-step parallels with merging: admissible e-merging functions reduce to weighted averages \citep{wang2024merging}, and the geometric proof in \citet{clerico2026merging} is based on a supporting-hyperplane argument, much like the one-step separation argument used here.
\paragraph{Acknowledgments.}
I am grateful to Sebastian Arnold for insightful and extensive discussions that helped inspire this work. I acknowledge the use of generative AI models (LLMs) during the preparation of this paper, for polishing and improving the exposition, and for exploring and simplifying possible proof ideas. I revised and edited all AI-generated material, and I take full responsibility for the content and correctness of the paper.
\bibliography{bib}
\bibliographystyle{abbrvnat}
\newpage

\appendix
\crefalias{section}{appendix}
\section{Details for the axial concentration example}\label{app:axial}
We let $\X = \mathbb{S}^2 = \{x\in\R^3\,:\,\|x\|_2=1\}$ and $\Phi(x) = xx^\top$. Since $x$ is a unit vector, $\Phi(x)$ is  positive semidefinite (psd) $3\times 3$ matrix with trace one. Given a probability measure $P$ on $\mathbb{S}^2$, we define $M_P = \E_P[\Phi(X)]$. $M_P$ is psd and with unit trace. We study the null hypothesis $\HhP$ that $M_P\in S$, where $$ S = \{M\in\R^{3\times 3}\text{ psd}\,:\,\Tr(M)=1\,,\,M\preceq\rho\,\Id\}\,.$$
Here we fixed $\rho\in[1/3,1)$ (as for $\rho<1/3$ the hypothesis would be empty and for $\rho\geq 1$ it would contain all probability measures on $\mathbb S^2$)  and $M\preceq\rho\,\Id$ means that $\rho\,\Id-M$ is psd. Note that $P\in\HhP$ if and only if
$$\E_P[(u^\top X)^2]\leq \rho$$ for every $u\in \mathbb{S}^2$. We extend this one-step null to the sequential case by requiring, as usual, the constraint to hold conditionally on the past. We note that in this setting $\X = \XP$. 

We endow the space of symmetric matrices with the trace inner product: $\langle A, M\rangle = \Tr(AM)$. Thus, the support function reads $$\sP(A) = \sup_{M\in S}\Tr(AM)\,.$$ Remark that $S\subseteq\aff\Phi(\XP)$ as we have included in $S$ the unit trace constraint. 

For a symmetric $A$, let $a_1\geq a_2\geq a_3$ denote its eigenvalues and $u_1, u_2, u_3$ the respective orthonormal eigenvectors. Then $A = a_1 u_1u_1^\top + a_2 u_2 u_2^\top + a_3 u_3 u_3^\top$. We define $a = a_2$, $b = a_1 - a_2\geq 0$, and $c = a_2 - a_3\geq 0$. Since $u_1u_1^\top + u_2u_2^\top + u_3u_3^\top = \Id$, we can rewrite $A$ as
$$A = a\,\Id + b\,u_1u_1^\top - c\,u_3u_3^\top\,.$$
We now split the analysis into two cases, $\rho\geq1/2$ and $\rho<1/2$, as these produce two different support functions $\sP$.

\paragraph{Case $\rho\geq 1/2$.} First, note that for any symmetric $A$ and $M$, we can write $\Tr(AM) = a\Tr(M) + b\,u_1^\top M u_1 - c\,u_3^\top M u_3$. If $M\in S$, then $\Tr(M)=1$, $u_1^\top M u_1\leq \rho$, and $u_3^\top M u_3\geq 0$. Since both $b$ and $c$ are non-negative we obtain the upper bound
$\sP(A)\leq a + \rho b$. We now show that when $\rho\geq 1/2$ this upper bound is achieved. Let $M^\star = \rho\,u_1u_1^\top + (1-\rho)\,u_2u_2^\top$, whose eigenvalues are $\rho$, $1-\rho$, and $0$. Then $M^\star$ is psd, $\Tr(M^\star) = 1$, and $M^\star \preceq \rho\Id$ (since $\rho\geq1/2$). On the other hand, we have that $\Tr(AM^\star) = a + b\rho$ and so we conclude that 
$$\sP(A) = a + b\rho = \rho a_1 + (1-\rho)a_2\,.$$

We know that the elements in $\EeP$ will be in the form $$e_A(x) = 1 + \Tr(A\Phi(x)) - \sP(A) = 1 + x^\top Ax - \sP(A)\,.$$ For this to be non-negative, we must have $$0\leq 1 + \min_{x\in \mathbb{S}^2}x^\top A x - \sP(A) = 1 + a_3 -\rho a_1 -(1-\rho)a_2\,.$$ This condition gives us the explicit form of $\LP$. Rearranging, we obtain 
$$\LP = \{A\,:\,b\geq 0\,,\,c\geq 0\,,\,\rho b+c\leq 1\}\,.$$

We will now show that $\Ee_{\LPmax}\subsetneq\EeP$. First, note that for any $x\in \mathbb{S}^2$, we have $x^\top Ax = a + b(u_1^\top x)^2 - c(u_3^\top x)^2$. Using this and the explicit expression for $\sP(A)$ we get that $e_A(x) = 1 + b((u_1^\top x)^2 -\rho) - c(u_3^\top x)^2$.
where $b\geq 0$, $c\geq 0$, and $\rho b+c\leq 1$. Now define $$\tilde A = a\,\Id + b\,u_1u_1^\top = \tilde a\,\Id + \tilde b\,u_1u_1^\top - \tilde c\,u_3u_3^\top\,,$$ with $\tilde a = a$, $\tilde b= b$, and $\tilde c=0$. Clearly, $\tilde b\geq 0$ and $\tilde c\geq 0$. Moreover, $\rho\tilde b +\tilde c\leq \rho b+c\leq1$, so $\tilde A\in\LP$ and $e_{\tilde A}$ is a valid e-variable. For every $x\in\mathbb{S}^2$ we have
$$e_{\tilde A}(x) = 1 + b((u_1^\top x)^2-\rho) \geq e_A(x)\,,$$ so $e_{\tilde A}$ majorises $e_A$ (strictly if $c>0$). In particular, a necessary condition for $A$ to be maximal is that $c = 0$. Next, we show that this is also sufficient. 

Fix $A\in\LP$ with $c=0$. Then $e_A(x)=1+b((u_1^\top x)^2-\rho)$. By \Cref{prop:evar}, it suffices to show that no affine e-variable can strictly majorise $e_A$. Let $e_B\in\EeP$ majorise $e_A$. We need to show that $e_B=e_A$. Consider the probability measure
$$P_\star=p\delta_{u_1}+(1-p)Q\,,$$
with $p = (3\rho-1)/2\in(0,1)$ and $Q$ the uniform measure on the sphere. Then $P_\star$ is fully supported on the sphere and has an atom on $u_1$. Moreover, we have
$$M_{P_\star} = \E_{P_\star}[XX^\top] = p\, u_1u_1^\top + (1-p)\,\Id/3\,,$$ whose eigenvalues are $\rho$, $(1-\rho)/2$, and $(1-\rho)/2$. In particular $P_\star\in\HhP$. Also note that $$\E_{P_\star}[e_A] = p(1+b-b\rho) + (1-p)(1 + b/3-b\rho) = 1 + 2pb/3 + b(1/3-\rho)=1\,,$$ where we used that $\E_{Q}[(u_1^\top X)^2] = u_1^\top\E_Q[X X^\top]u_1 = 1/3$. Therefore, since $e_B$ majorises $e_A$, we have
$$0\leq\E_{P_\star}[e_B-e_A] \leq 1-\E_{P_\star}[e_A] = 0\,.$$
Since $e_B-e_A$ is continuous and non-negative and $P_\star$ is fully supported, we conclude that $e_B = e_A$ everywhere on $\mathbb S^2$. This shows that $e_A$ is maximal. Therefore, we conclude that
$$\LP^\star = \{A\,:\,b\in[0,1/\rho]\,,\,c=0\} = \{A\,:\,a_1-a_2\in[0,1/\rho]\,,\,a_3=a_2\}\,.$$

We remark that the value of $a$ is irrelevant to the corresponding e-variable. Indeed, one can easily check that $e_A=e_{A+\xi\Id}$ for every $\xi\in\R$. Thus, distinct elements of $\LPmax$ may correspond to the same e-variable, and we can restrict attention to representatives with $a=0$ without any loss. Since maximality is equivalent to $c=0$, these representatives take the form $A=\beta uu^\top$, where $u\in\mathbb S^2$ and $0\leq \beta\leq1/\rho$. We can therefore parametrise the maximal  e-variables by a direction $u$ and a scalar $\beta$, obtaining
$$\Ee_{\LPmax}=\left\{e_{\beta,u}\,:\,u\in\mathbb S^2\,,\,\beta\in[0,1/\rho]\right\}\,,$$
with $e_{\beta,u}(x)= 1+\beta((u^\top x)^2-\rho)$.
\paragraph{Case $\rho<1/2$.} Recall that for any symmetric $A$ we write $A = a\,\Id + b\,u_1u_1^\top - c\,u_3u_3^\top$, with $b\geq 0$ and $c\geq 0$. We also know that for any $M\in S$ we have 
$\Tr(AM) \leq a + b\rho - c\,u_3^\top M u_3$, since $\Tr(M)=1$ and $u_1^\top M u_1\leq\rho$. Now, using that $u_2^\top M u_2\leq\rho$, we find
$u_3^\top M u_3 = \Tr(M) - u_1^\top M u_1 - u_2^\top M u_2 \geq 1-2\rho$, where the right hand side is larger than zero since $\rho<1/2$. So, we obtain the upper bound $\sP(A)\leq a + \rho b-(1-2\rho)c$. This bound is achieved. Indeed, define 
$M^\star = \rho\,u_1u_1^\top + \rho\,u_2u_2^\top + (1-2\rho)\,u_3u_3^\top$. Its eigenvalues are $\rho$, $\rho$, and $1-2\rho$, so $M^\star\in S$. Moreover, $\Tr(AM^\star) = a+\rho b-(1-2\rho)c$, thus
$$\sP(A) = a + \rho b -(1-2\rho)c\,.$$
Now, imposing the non-negativity of $1+x^\top A x-\sP(A)$, we find that $0\leq \min_{x\in\mathbb S^2}e_A(x) = 1 + a_3 - \sP(A) = 1 - \rho b-2\rho c$, and so
$$\LP = \{A\,:\,b\geq 0\,,\,c\geq 0\,,\,\rho(b+2c)\leq 1\}\,.$$
The e-variable $e_A$ takes the form
$$e_A(x) = 1 + b((u_1^\top x)^2 - \rho) + c(1-2\rho-(u_3^\top x)^2)\,.$$
Again there is no dependency in $a$, so $e_A=e_{A+\xi\Id}$ for any $\xi\in\R$. However, now the term proportional to $c$ is not always negative, so we do not have to zero it to get a maximal e-variable. Indeed, it turns out that $\LP=\LPmax$ for $\rho\in[1/3,1/2)$, which we will now show. 

Fix $A\in\LP$ and let $$P_\star = p\delta_{u_1} + p\delta_{u_2} + (1-2p)Q\,,$$ with $p = 3\rho-1$ and again $Q$ denoting the unicorn distribution on $\mathbb S^2$. It is easily checked that $P_\star\in\HhP$ and direct calculation shows that $\E_{P_\star}[e_A]=1$. Now, let $B\in\LP$ such that $e_B$ majorises $e_A$. Then, $$0\leq \E_{P_\star}[e_B-e_A]\leq 1-\E_{P_\star}[e_A]=0\,,$$ so $e_B=e_A$ $P_\star$-almost surely. Since both $e_A$ and $e_B$ are continuous and $P_\star$ is fully supported, we conclude that $e_B=e_A$. So each element in $\EeP$ is maximal.

Since the dependence on $a$ is irrelevant, we can choose representatives with smallest eigenvalue equal to zero, without any loss. We can therefore parametrise the maximal affine e-variables by two orthogonal directions $u,v\in\mathbb S^2$ and two non-negative scalars $\beta,\gamma$, obtaining
$$\Ee_{\LPmax}=\left\{e_{\beta,\gamma,u,v}\,:\,u,v\in\mathbb S^2\,,\,u^\top v=0\,,\,\beta,\gamma\geq0\,,\,\beta+\gamma\leq1/\rho\right\}\,,$$
with $e_{\beta,\gamma,u,v}(x)=1+\beta((u^\top x)^2-\rho)+\gamma((v^\top x)^2-\rho)$.

\section{Tools from convex analysis}\label{app:convex}
\subsection{Preliminaries}
A subset $C$ of $\R^m$ is \emph{convex} if, for any $x,x'\in C$ and  $\alpha\in[0,1]$, one has $\alpha x + (1-\alpha)x'\in C$. We remark that in general a convex set in $\R^m$ need not be Borel (for instance, the union of the unit open disk in $\R^2$ with a non-Borel subset of the unit circle in $\R^2$ is convex but non-Borel). A convex subset $F\subseteq C$ is a \emph{face} of the convex set $C$ if, whenever $x,x'\in C$ and $\alpha\in(0,1)$ are such that $\alpha x+(1-\alpha)x'\in F$, then $x,x'\in F$. 

Given a subset $S\subseteq\R^m$, its \emph{convex hull} $\conv S$ is the smallest convex set containing $S$. Equivalently,
$$\conv S = \left\{\sum_{i=1}^n \alpha_i x_i\,:\, n\geq 1\,,\; x_1,\dots,x_n\in S\,,\; \alpha_1,\dots,\alpha_n\geq 0\,,\; \sum_{i=1}^n \alpha_i=1\right\}\,.$$
The \emph{affine hull} of $S$, denoted by $\aff S$, is the smallest affine space containing $S$:
$$\aff S = \left\{\sum_{i=1}^n \beta_i x_i\,:\, n\geq 1\,,\; x_1,\dots,x_n\in S\,,\; \beta_1,\dots,\beta_n\in\R\,,\; \sum_{i=1}^n \beta_i=1\right\}\,.$$
Clearly, $S\subseteq\conv S\subseteq\aff S$ and $\aff(\conv S) = \aff S$. The \emph{relative interior} of a convex set $C\subseteq\R^m$, denoted by $\relint C$, is the interior of $C$ with respect to the relative topology on $\aff C$. If $C\neq\emptyset$, then $\relint C\neq\emptyset$. 

We shall use the following two classical results without proof. The first is Carathéodory's theorem (Theorem 17.1 in \citealt{rockafellar1970convex}).
\begin{theorem}[Carathéodory]\label{thm:caratheodory}
    Let $S\subseteq\R^m$. Then every point of $\conv S$ can be written as a convex combination of at most $m+1$ points of $S$, that is, as $\sum_{i=1}^k \alpha_i x_i$ for some $k\leq m+1$, $x_1,\dots,x_k\in S$, $\alpha_1,\dots,\alpha_k\geq 0$, and $\sum_{i=1}^k \alpha_i=1$.
\end{theorem}
The second is the separating hyperplane theorem (Theorem 11.3 in \citealt{rockafellar1970convex}).
\begin{theorem}[Separating hyperplane]\label{thm:proper_separation}
    Let $C_1$ and $C_2$ be non-empty convex sets in $\R^n$. If their relative interiors are disjoint, meaning $\relint C_1 \cap \relint C_2 = \emptyset$, then there exists a hyperplane that properly separates $C_1$ and $C_2$. \\
    Specifically, there exists a non-zero vector $a \in \R^n$ and a scalar $c \in \R$ such that, for every $x\in C_1$ and every $y\in C_2$,
    $$ a \cdot x \leq c \leq a \cdot y\,,$$
    and $C_1$ and $C_2$ are not both contained within the  hyperplane $\{z \in \R^n : a \cdot z = c\}$.
\end{theorem}
\subsection{Technical lemmas}
The next lemma  is often stated under the additional assumption that the convex set is Borel (see, e.g., Theorem 10.2.6 in \citealt{dudley2002real}). Since we do not impose this assumption here, we include a proof for completeness.

\begin{lemma}\label{lemma:convexhull}
    Let $\X$ be a standard Borel space and $\Phi: \X \to \R^m$ be a Borel map. Denote by $\PR_\Phi$ the set of all Borel probability measures $P$ on $\X$ such that $\E_P[\|\Phi\|_1] < \infty$. Then, the set of all possible expected values of $\Phi$ is exactly the convex hull of its range:
    $$ \{ \E_P[\Phi] \,:\, P \in \PR_\Phi \} = \conv(\Phi(\X))\,. $$
\end{lemma}
\begin{proof}
    Let $A_\Phi = \{ \E_P[\Phi] \,:\, P \in \PR_\Phi \}$ and let $C = \conv(\Phi(\X))$. First we show that $C\subseteq A_\Phi$. Take any $y \in C$. By definition of convex hull, there exist points $x_1,\dots,x_k\in\X$ and coefficients $\alpha_1,\dots,\alpha_k\geq 0$ with $\sum_{i=1}^k\alpha_i=1$ such that $y=\sum_{i=1}^k \alpha_i \Phi(x_i)$. Define $P=\sum_{i=1}^k \alpha_i \delta_{x_i}$. Then $P\in\PR_\Phi$ and $\E_P[\Phi]=y$. Hence $y\in A_\Phi$, and so $C\subseteq A_\Phi$.

    For the reverse inclusion, we proceed by induction on the dimension $m$. If $m=0$, then $\Phi$ is identically null, so it must be that $C=\{0\}$ and the result is trivial. If $m\geq 1$, suppose that the result holds for $m-1$. Fix $P\in\PR_\Phi$. We note that we can assume that $\E_P[\Phi]=0$ with no loss of generality (simply by replacing $\Phi$ with $\Phi-\E_P[\Phi]$). So, we need to show that $0\in C$. Assume by contradiction that $0\notin C$. Then, by the hyperplane separation theorem \Cref{thm:proper_separation} there is a non-zero $v\in\R^m$  such that $v\cdot z\leq0$, for every $z\in C$. Let us show that $P$ must be supported in $\X' = \Phi^{-1}(\Pi)$, where $\Pi$ is the hyperplane $\{z\in\R^m\,:\,v\cdot z=0 \}$. Indeed, if this was not the case, there would be a Borel $B\subseteq\X$ such that $P(B)>0$ and $v\cdot \Phi(x) \neq 0$ for all $x\in B$. Since $v\cdot \Phi(x)\leq 0$ for all $x\in\X$, this  would yield that $0=v\cdot\E_P[\Phi] < 0$, which cannot be. Now, we can apply the induction step replacing $\X$ with $\X'$, which is a Borel set, and $\Phi$ by its restriction to $\X'$, which is valued in $\Pi\cong\R^{m-1}$. Then, $0\in\conv\Phi(\X') \subseteq C$,  a contradiction.
\end{proof}

\begin{lemma}\label{lemma:face}
    Let $S\subseteq\R^n$ be a  convex set. Let $T$ be a  non-empty subset of $\R^n$, and denote by $C$ its convex hull. Assume that $C\cap S\neq \emptyset$. Then, there exists a face $F$ of $C$ such that $C\cap S\subseteq F$, $S\cap \relint F \neq\emptyset$, and  $F = \conv(T\cap \aff F)$. 
\end{lemma}
\begin{proof}
    Let $G = C\cap S$. $G$ is convex and non-empty, hence its relative interior is also non-empty and we can fix $y\in\relint G$. Let $F$ be the minimal face of $C$ containing $y$. Let us show that $G\subseteq F$. Take any $z \in G$. Because $y \in \relint G$ and $G$ is convex, there exists $w \in G$ and $\alpha \in (0, 1)$ such that $y = \alpha z + (1-\alpha)w$. Since $z, w \in G \subseteq C$, the point $y$ is a convex combination of points in $C$, with strictly positive coefficients. Since $y \in F$ and $F$ is a face of $C$, the defining property of faces implies that both $z$ and $w$ must belong to $F$.  Since $z \in G$ was arbitrary, this establishes that $G \subseteq F$.
    
    Because $F$ is minimal, we must have $y\in\relint F$ (if $y$ was on the relative boundary of $F$, then it would be contained in a smaller face). Since $y\in G\subseteq S$, and $y\in\relint F$, we immediately get that $S\cap \relint F\neq \emptyset$. 

    We are left with proving that $F=\conv(T\cap \aff F)$. Fix $x\in F$. Since $F\subseteq C$ and $C = \conv T$, $x$ can be expressed as a finite convex combination $x = \sum_{i=1}^k\lambda_i t_i$, where $t_i\in T$, $\lambda_i>0$ and $\sum_{i=1}^k\lambda_i = 1$. Again, the fact that all the $\lambda_i$ are strictly positive and $x\in F$ implies that all the $t_i$ are in $F$ (since $F$ is a face). Hence, $x$ is a convex combination of points in $T\cap F$, and so $x\in\conv(T\cap F)$. Since the choice of $x\in F$ was arbitrary, $F \subseteq \conv(T\cap F)\subseteq \conv(T\cap\aff F)$. 
    
    For the reverse inclusion, first note that $\conv(T\cap\aff F)\subseteq\conv(C\cap\aff F)$. Now, since $F$ is convex, it is enough to show that $C\cap\aff F\subseteq F$. Pick any $z\in C\cap\aff F$ and $y\in\relint F$. Then there is $u\in F$ and $\alpha\in(0,1)$ such that $y = \alpha z + (1-\alpha)u$. In particular, because $z$ and $u$ are both in $C$ and $F$ is a face, $z$ and $u$ are in $F$. As the choice of $z$ was arbitrary, we conclude.
\end{proof}
The next lemma  is a simple refining of Carathéodory's theorem under a downward-closure in the last coordinate.
\begin{lemma}\label{lemma:caradown}
    Let $A\subseteq\R^m\times\R$ be a non-empty set, and assume that whenever $(y,z)\in A$ and $z'\leq z$, then $(y,z')\in A$. Then, every point of $\conv A$ can be written as a convex combination of at most $m+1$ points of $A$.
\end{lemma}
\begin{proof}
    Let $(y,z)\in\conv A$. By \Cref{thm:caratheodory} in $\R^{m+1}$, we can write $(y,z)$ as a convex combination of at most $m+2$ points of $A$. If it requires $m+1$ or fewer points, we are done. 
    
    Thus, suppose $(y,z)=\sum_{i=1}^{m+2}(\alpha_i y_i,\alpha_i z_i)$ for points $(y_i,z_i)\in A$ and weights $\alpha_i>0$ summing to $1$. The $m+2$ points $y_1,\dots,y_{m+2}$ lie in $\R^m$, so they must be affinely dependent. Therefore, there exist real numbers $\gamma_1,\dots,\gamma_{m+2}$, not all zero, such that
    $$\sum_{i=1}^{m+2}\gamma_i y_i = 0 \qquad \text{and} \qquad \sum_{i=1}^{m+2}\gamma_i = 0\,.$$
    By replacing $\gamma$ with $-\gamma$ if necessary, we may assume without loss of generality that $\sum_{i=1}^{m+2}\gamma_i z_i \geq 0$. 
    
    Since $\sum_{i=1}^{m+2}\gamma_i=0$ and the $\gamma_i$ are not all zero, at least one $\gamma_i$ is strictly negative. Define
    $$t = \min_{\gamma_i < 0} \frac{\alpha_i}{-\gamma_i}\,.$$
    Because all $\alpha_i>0$, we have $t>0$. Let $j$ be an index where this minimum is attained. We define new weights $\beta_i = \alpha_i + t\gamma_i$. By our choice of $t$, we have $\beta_i \geq 0$ for all $i$, $\beta_j = 0$, and $\sum_{i=1}^{m+2}\beta_i = 1$. 
    
    Using these new weights, we have $\sum_{i=1}^{m+2}\beta_i y_i = y $. Let
    $$z^\star = \sum_{i=1}^{m+2}\beta_i z_i = z + t\sum_{i=1}^{m+2}\gamma_i z_i\,.$$
    Since $t>0$ and $\sum_{i=1}^{m+2}\gamma_i z_i \geq 0$, it follows that $z^\star \geq z$. 
    
    If $z^\star=z$, then $\sum_{i=1}^{m+2}\beta_i(y_i,z_i)$ is a convex combination of $(y,z)$ using at most $m+1$ points (since $\beta_j=0$), and we are done.
    
    If $z^\star>z$, pick any index $k$ such that $\beta_k>0$, and define $z_k' = z_k - \frac{z^\star-z}{\beta_k}$. Since $z_k'<z_k$, we have $(y_k,z_k')\in A$ by the downward closure assumption. Replacing $(y_k,z_k)$ with $(y_k,z_k')$ shifts the $z$-coordinate of the combination exactly to $z$:
    $$\sum_{i\neq k}\beta_i z_i + \beta_k z_k' = z^\star - (z^\star-z) = z\,.$$
    This yields $(y,z)$ as a convex combination of at most $m+1$ points of $A$.
\end{proof}

\section{Technicalities}\label{app:tech}

    \subsection{Effective domain}
    \begin{lemma}\label{lemma:exF}
        Let $\HhP$ be non-empty and let $C=\conv\Phi(\X)$. Then $C\cap S\neq\emptyset$ and, if $\X = \XP$, then $S \cap \relint C  \neq \emptyset$. Moreover, there is a face $F$ of $C$ such that $C\cap S\subseteq F$, $S\cap\relint F\neq\emptyset$,  $F = \conv(\Phi(\X)\cap\aff F)$, and $\XP = \Phi^{-1}(\aff F)$. 
    \end{lemma}
    \begin{proof}
        Let $T$ denote the range of $\Phi$, so $C=\conv T$. Note that $C = \{\E_P[\Phi]\,:\,P\in\PR_\Phi\}$ (\Cref{lemma:convexhull}) and the condition $P\in\HhP$ is equivalent to saying that $\E_P[\Phi]\in C\cap S$. So, $\HhP\neq\emptyset$ implies that $C\cap S\neq\emptyset$. In particular, by \Cref{lemma:face}, there is a face $F$ of $C$ such that $C\cap S\subseteq F$, $S\cap \relint F \neq\emptyset$, and $F = \conv(T\cap\aff F)$. We will now show that $\XP = \Phi^{-1}(\aff F)$. 
        
        First, let us show that $\XP\subseteq\Phi^{-1}(\aff F)$. Fix $x\in\XP$. By definition, there is a $P\in\HhP$ such that $P(\{x\})=\alpha>0$. We can write $P$ as $P = \alpha\delta_x + (1-\alpha)Q$ where $Q\in\PR_\Phi$. We have $\E_P[\Phi] = \alpha\Phi(x) + (1-\alpha)\E_Q[\Phi]$. 
        Because $\E_P[\Phi]\in C\cap S\subseteq F$, $F$ is a face of $C$, and both $\Phi(x)$ and $\E_Q[\Phi]$ belong to $C$, it must be that $\Phi(x)\in F\subseteq \aff F$. So $x\in \Phi^{-1}(\aff F)$. 
    
        To show the reverse inclusion, fix $x\in\X$ such that $\Phi(x)\in\aff F$. Since, $F=\conv(T\cap\aff F)$, we have $\Phi(x)\in F$. Fix $y\in S\cap\relint F$, which is non-empty by construction. Since $y\in\relint F$ and $\Phi(x)\in F$, we can find $\alpha\in(0,1)$ and $z\in F$ such that $y = \alpha \Phi(x) + (1-\alpha)z.$ Since $z\in F\subseteq C$, it can be represented as a finite convex combination of points in $T$. As a consequence, there is a measure $Q\in\PR_\Phi$ such that $\E_Q[\Phi] = z$. We can define a probability measure $P$ as
        $P = \alpha\delta_x + (1-\alpha)Q.$
        By construction, we have that $\E_P[\Phi] = y$, and so $P\in\HhP$ because $y\in S$. Moreover, $P(\{x\})\geq \alpha>0$, and hence $x\in\XP$. 
    
        Finally, we are left to show that if $\X=\XP$, then $S\cap\relint C\neq\emptyset$. So, assume that $\X=\XP$. Then $\Phi^{-1}(\aff F ) = \X$, and so  $\Phi(\X)\subseteq\aff F$. Since $\aff C = \aff\Phi(\X)$, we get that $\aff C\subseteq\aff F$. On the other hand, $\aff F\subseteq \aff C$ because $F\subseteq C$, so $\aff C = \aff F$. Now, since $F$ is a face of $C$, the fact that $\aff F=\aff C$ implies that $F=C$. As $S \cap \relint F \neq \emptyset$, we get $S \cap \relint C \neq \emptyset$.
    \end{proof}
    \begin{reflem}{lemma:XP}
        $\XP$ is a Borel subset of $\X$ and $P(\XP)=1$ for every $P\in\HhP$.
    \end{reflem}
    \begin{proof}
        If $\HhP=\emptyset$, $\XP=\emptyset$ and the claim is trivial. So, let us assume that $\HhP\neq\emptyset$. First, let us show that $\XP$ is a Borel set. By \Cref{lemma:exF} there is a face $F$ of $\conv\Phi(\X)$ such that  $\XP = \Phi^{-1}(\aff F)$. Since $\aff F$ is a closed set in $\R^m$ and $\Phi$ is a Borel map, $\XP$ is a Borel set.

        We now show that $P(\XP)=1$ for all $P\in\HhP$. Let $\XP^c=\X\setminus\XP$. Suppose for contradiction that $\alpha = P(\XP^c)>0$ for some $P\in\HhP$. We can decompose $P$ as $P = \alpha P' + (1-\alpha)P''$, where $P'(\XP^c) = 1$ and $P''(\XP) = 1$. Since $\E_{P'}[\Phi]\in\conv(\Phi(\XP^c))$, we can find a finitely supported probability measure $Q' = \sum_{i=1}^k \lambda_i\delta_{x_i}$, with $x_i\in\XP^c$ and $\lambda_i>0$, such that $\E_{Q'}[\Phi] = \E_{P'}[\Phi]$. It is quickly checked that $Q = \alpha Q' + (1-\alpha)P''$ is in $\HhP$. Yet, $Q(\{x_1\}) = \alpha\lambda_1>0$, which is a contradiction since $x_1\notin\XP$.
    \end{proof}

\subsection{Relatively open reduction}
\begin{lemma}\label{lemma:ropen}
    Let $S\subseteq\R^m$ be convex and assume  $\HhP\neq\emptyset$. Let $A = \aff\Phi(\XP)$ and $S'=\relint(S\cap A)$. Then, $S'$ is non-empty and convex, and $\Hh_{\Phi,S'}\neq\emptyset$. Moreover, $\XP = \X_{\Phi,S'}$ and $\sP = \sigma_{\Phi,S'}$. In particular, $\LP = \Lambda_{\Phi,S'}$. 
\end{lemma}
\begin{proof}
    Let $C=\conv\Phi(\X)$. By \Cref{lemma:exF}, there exists a face $F$ of $C$ such that $C\cap S \subseteq F$, $S\cap\relint F\neq\emptyset$, $F = \conv(\Phi(\X)\cap\aff F)$, and $\XP=\Phi^{-1}(\aff F)$. In particular, $\Phi(\XP) = \Phi(\X)\cap\aff F$, and so
    $$A = \aff\Phi(\XP) = \aff(\Phi(\X)\cap\aff F) = \aff F\,.$$ Since $S\cap\relint F\neq\emptyset$ and $\relint F\subseteq\aff F = A$, we have $S\cap A\neq\emptyset$. Because $S\cap A$ is convex, its relative interior $S'$ is non-empty and convex. 

    We now show that $\Hh_{\Phi,S'}\neq\emptyset$. Take $y\in S\cap\relint F$. Since $y\in S\cap A$, it belongs to the closure of $S'$ in $A$ and hence every open neighbourhood of $y$ in $A$ intersects $S'$. Since $\relint F$ is an open neighbourhood of $y$ in $A$, we obtain that $\relint F\cap S'\neq\emptyset$. Since $F\subseteq C$, this implies $C\cap S'\neq\emptyset$, which is enough to ensure that $\Hh_{\Phi,S'}$ is non-empty by \Cref{lemma:convexhull}.

    Because $\Hh_{\Phi,S'}\neq\emptyset$, we can apply \Cref{lemma:exF} with $S'$ in place of $S$. There exists a face $F'$ of $C$ such that $C\cap S'\subseteq F'$, $S'\cap\relint F'\neq\emptyset$, and $\X_{\Phi,S'} = \Phi^{-1}(\aff F')$.
    We will show that $F = F'$. First, take $z\in S'\cap\relint F$, which is non-empty as shown above. Since $z\in S'$ and $z\in F\subseteq C$, we have $z\in C\cap S'\subseteq F'$. Thus, $F'$ is a face of $C$ that intersects the relative interior of the face $F$. Since $F$ and $F'$ are both faces of $C$, this implies $F\subseteq F'$ (see, e.g., Theorem 18.1 in \citealt{rockafellar1970convex}).
    Conversely, take $w\in S'\cap\relint F'$, which is non-empty by the choice of $F'$. Because $w\in S'\subseteq A = \aff F$ and $w\in F'\subseteq C$, we have $w\in C\cap\aff F$. Since $F$ is a face of $C$, we have $C\cap\aff F = F$, and so $w\in F$. Thus, the face $F$ intersects the relative interior of the face $F'$, which implies $F'\subseteq F$. Therefore, $F = F'$. So,
    $$\X_{\Phi,S'} = \Phi^{-1}(\aff F') = \Phi^{-1}(\aff F) = \XP$$ by \Cref{lemma:exF}.

    Now, let us show that $\sP = \sigma_{\Phi,S'}$. By definition, $\sP(\lambda) = \sup_{y\in S\cap A}\lambda\cdot y$. Similarly, since $\XP = \X_{\Phi,S'}$, we have $\sigma_{\Phi,S'}(\lambda) = \sup_{y\in S'\cap A}\lambda\cdot y$. We note that the supremum of any linear functional over $S\cap A$ is equal to its supremum over its relative interior $S'$. As $S' = S'\cap A$, we get $\sP = \sigma_{\Phi,S'}$. It follows immediately that  $\LP = \Lambda_{\Phi,S'}$.
\end{proof}

\section[Proof of Lemma 4]{Proof of \Cref{lemma:dom}}\label{app:proof_dom}
\begin{reflem}{lemma:dom}
    Assume $\HhP\neq\emptyset$ and $\X=\XP$. Let $E$ be an e-process for $\HhP^\infty$. Then there is $\lambda\in\LP$ that dominates $E$ at level $1$.
\end{reflem}

\begin{proof}
    Let $C = \conv\Phi(\X)$. We define $U\subseteq \R^{m+1}$ as $$U = (S\cap\aff C)\times(1,+\infty)$$ and $V\subseteq\R^{m+1}$ as $$V = \{(\Phi(x),z)\in\R^{m+1}\,:\,x\in\X\,,\;\exists Q\in\HhP^\infty\,,\;\exists\tau\in\T_\star\,,\;z\leq\E_Q[E^x_\tau]\}\,.$$
    The fact that $\HhP\neq\emptyset$ implies that $S\cap\aff C\neq\emptyset$ (cf.~\Cref{lemma:convexhull} in \Cref{app:convex}), so $U\neq\emptyset$. $V$ is also non-empty, since $\HhP^\infty$ is non-empty and so $(\Phi(x),0)\in V$ for any $x\in \XP=\X$. Moreover, $U$ is a convex set. We let $K = \conv V$. 

    Let us now show by contradiction that $K\cap U=\emptyset$. Assume that this was not the case, and fix $(y,z)\in K\cap U$. Since $(y,z)\in K$, a slight refinement of Carathéodory theorem (\Cref{lemma:caradown} in \Cref{app:convex}) ensures that we can find $m'\in[1: m+1]$ distinct points $\bar x_1, \dots, \bar x_{m'}$ in $\X$, measures $Q_1,\dots, Q_{m'}$ in $\HhP^\infty$, stopping times $\tau_1,\dots,\tau_{m'}$ in $\T_\star$, and non-negative weights $w_1,\dots,w_{m'}$ summing up to $1$, such that 
    $$y = \sum_{i=1}^{m'}w_i\Phi(\bar x_i)\,;\qquad\qquad z \leq\sum_{i=1}^{m'}w_i\E_{Q_i}[E^{\bar x_i}_{\tau_i}]\,.$$
    Let $P = \sum_{i=1}^{m'}w_i\delta_{\bar x_i}$. Then, $\E_P[\Phi] = y\in S$, and so $P\in\HhP$. We may now define a probability measure $Q_\star$ on $\X^\infty$ by requiring that when $X^\infty\sim Q_\star$, its first marginal is $X_1\sim P$ and, for each $i=1,\dots, m'$, the conditional law of $X^{2:\infty}$ given $X_1=\bar x_i$ is $Q_i$. It is then straightforward that $Q_\star\in\HhP^\infty$. We also define $\tau_\star\in\T_\star$ as $\tau_\star(x^\infty) = 1 + \tau_i(x^{2:\infty})$ if $x_1 = \bar x_i$ for some $i$, and $1$ otherwise. With this construction, we get
    $$\E_{Q_\star}[E_{\tau_\star}] = \sum_{i=1}^{m'} w_i \E_{Q_i}[E^{\bar x_i}_{\tau_i}]\geq z\,.$$
    Yet, since $(y,z)\in U$, it must be that $z>1$, which contradicts the fact that $E$ is an e-process. We thus conclude that $K\cap U= \emptyset$. 

    \Cref{lemma:convexhull} implies that the projection of $K$ onto the first $m$ coordinates is $C$. Moreover, since $K$ is downward closed in the last coordinate, we have $\aff K = \aff C\times\R$. We denote as $L_C$ the linear subspace of $\R^m$ parallel to $\aff C$. Since $U\subseteq\aff K$ and $K\cap U=\emptyset$, we can  apply the separating hyperplane theorem (see \Cref{thm:proper_separation} in \Cref{app:convex}) in the affine subspace $\aff K$, and find a non-zero pair $(a,b)\in L_C\times\R$ and a scalar $c\in\R$ such that, for all $(y,z)\in K$ and $(y',z')\in U$, 
    \begin{equation}\label{eq:sep}a\cdot y + bz \leq c\leq a\cdot y' + bz'\,.\end{equation}
    Because $z'$ can be arbitrarily large in $U$ for every valid $y'$, we must have $b\geq 0$. We now show that $b>0$. As we are assuming that $\X=\XP$, we can find a point $y_\star\in S\cap\relint C$ (cf.~\Cref{lemma:face} in \Cref{app:convex}). Suppose that $b=0$. Then \eqref{eq:sep} implies that  $a\cdot y\leq a\cdot y_\star$ for every $y\in C$. Fix any $y\in C$. As $y_\star\in\relint C$, there is $\varepsilon>0$ such that $y' = y_\star + \varepsilon(y_\star - y)\in C$. In particular, $a\cdot y_\star + \varepsilon(a\cdot y_\star - a\cdot y) = a\cdot y'\leq a\cdot y_\star$, and so $a\cdot y_\star = a\cdot y$. Since this holds for all $y\in C$, $a$ must be orthogonal to $L_C$. Yet, $a\in L_C$, and so $a = 0$. But $(a,b) = (0,0)$ contradicts the requirement that $(a,b)$ is a non-zero separator. 

    Once established that $b>0$, we  define $\lambda =- a/b\in L_C$. Dividing \eqref{eq:sep} by $b$,
    \begin{equation}\label{eq:sep_}
        z\leq\lambda\cdot y + \inf_{(y',z')\in U}(-\lambda\cdot y' + z') = 1 + \lambda\cdot y - \sP(\lambda)
    \end{equation}
    for every $(y,z)\in K$. In particular, since $(\Phi(x),0)\in K$ (for any $x\in\X$), we have $\lambda\in\LP$.
    
    Now, fix $x\in\X$, $Q\in\HhP^\infty$, and $\tau\in\T_\star$. Since $\X=\XP$, there is $P\in\HhP$ such that $P(\{x\})=\varepsilon>0$. Then, $P\otimes Q\in\HhP^\infty$ and so the fact that $E$ is an e-process requires that $$\E_Q[E_\tau^x]\leq \E_{P\otimes Q}[E_{1+\tau}]/\varepsilon\leq 1/\varepsilon<\infty\,.$$
    In particular, $(\Phi(x), \E_Q[E^x_\tau])\in V\subseteq K$. Then, \eqref{eq:sep_} yields
    $$\E_Q[E^x_\tau]\leq 1+\lambda\cdot\Phi(x) - \sP(\lambda) = \e_\lambda(x)\,,$$ which concludes the proof. 
\end{proof}

\section[Proof of Lemma 6]{Proof of \Cref{lem:mcc}}\label{app:proof_mcc}
\begin{reflem}{lem:mcc}
    Assume $\HhP\neq\emptyset$. $\LPmax$ is Borel, and there exists a Borel map $J:\LP\to\LPmax$ such that $\e_{J(\lambda)}\geq\e_\lambda$ for every $\lambda\in\LP$, with $J(\lambda)=\lambda$ whenever $\lambda\in\LPmax$. 
\end{reflem}

\begin{proof}
Let $C=\conv\Phi(\XP)$, $A=\aff C$, and $V=A-A$, the linear space parallel to $A$. Since $\HhP\neq\emptyset$, \Cref{lemma:exF}, applied to the effective domain $\XP$, ensures that there exists $z_\star\in S\cap\relint C$. Fix $r>0$ such that $z_\star+rB_V\subseteq C$, where $B_V$ denotes the closed unit ball of $V$.

For every $\lambda\in\LP$, define
$$D_\lambda=\left\{\eta\in\LP\cap(\lambda+V)\,:\,\inf_{x\in\XP}(\eta-\lambda)\cdot\Phi(x)\geq\sP(\eta)-\sP(\lambda)\right\}\,.$$
To understand this definition, note that for every $\lambda,\eta\in\LP$ and $x\in\XP$,
$$\e_\eta(x)-\e_\lambda(x)=(\eta-\lambda)\cdot\Phi(x)-\sP(\eta)+\sP(\lambda)\,.$$
Since both e-variables equal $+\infty$ outside $\XP$, the inequality in the definition of $D_\lambda$ is equivalent to $\e_\eta$ majorising $\e_\lambda$. Thus, $\eta\in D_\lambda$ precisely when $\eta-\lambda\in V$ and $\e_\eta$ majorises $\e_\lambda$.  The restriction to $\lambda+V$ does not exclude any affine dominator. Indeed, fix any $\eta' \in \LP$ and write $\eta' = \eta + \xi$, with $\eta\in\lambda + V$ and $\xi\in V^\perp$. Then $\xi\cdot z=\xi\cdot z_\star$ for every $z\in A$. In particular, $\xi\cdot \Phi(x) = \xi\cdot z_\star$ and  $\sP(\eta')=\sP(\eta)+\xi\cdot z_\star$. It follows that $\eta\in\LP$ and  $e_{\eta'}=\e_\eta$. Thus, every affine e-variable majorising $\e_\lambda$ has a representative with coefficient in $D_\lambda$.

Define
$$F(\eta)=1+\eta\cdot z_\star-\sP(\eta)$$
for $\eta\in\LP$. Since $\eta\mapsto\sP(\eta)$ is lower semi-continuous (as the supremum of a family of continuous functions), $F$ is upper semi-continuous. Moreover, for $\eta\in\LP$ we have that $1+\eta\cdot z-\sP(\eta)\geq 0$ for $z\in\Phi(\XP)$, and by affinity for any $z\in C$. So $F(\eta)\geq 0$. Note also that since $z_\star\in S\cap C\subseteq S\cap A$, by definition of $\sP$ we have that $F(\eta)\leq 1$. Hence, $F$ takes values in $[0,1]$. 

Now, we show that $D_\lambda$ is non-empty and compact. It is non-empty because $\lambda\in D_\lambda$. Now, because $\sP$ is lower-semicontinuous, $\LP$ is closed, $\eta\mapsto e_\eta(x)$ is upper semi-continuous for any fixed $x\in\XP$, and so the set $\{\eta\in\LP\,:\,e_\eta\text{ majorises $e_\lambda$}\}$ is also closed. Taking the intersection with $\lambda+V$ we see that $D_\lambda$ is closed. To prove boundedness, fix $\eta\in D_\lambda$. Since $\eta-\lambda\in V$ and $\e_\eta$ majorises $\e_\lambda$, the affine function $z\mapsto (\eta-\lambda)\cdot (z-z_\star)+F(\eta)-F(\lambda)$ is non-negative on $\Phi(\XP)$, and hence throughout its convex hull $C$. We then obtain that $z = z_\star - r(\eta-\lambda)/\|\eta-\lambda\|$ (with the convention $0/0=1$), which is in $z_\star + r B_V\subseteq C$, we obtain
\begin{equation}\label{eq:mcc}r\|\eta-\lambda\|\leq F(\eta) - F(\lambda)\leq1\,,\end{equation}
where the last inequality holds since we have shown that $F$ takes values in $[0,1]$. So, $D_\lambda$ is contained in the closed ball of radius $1/r$ centred in $\lambda$.

The upper semi-continuity of $F$ and the compactness of $D_\lambda$ ensure that, for every $\lambda\in\LP$, the maximum
$$j(\lambda)=\max_{\eta\in D_\lambda}F(\eta)$$
exists. We now show that $j$ is Borel. First, define the set $$G = \{(\lambda,\eta)\in\LP\times\LP\,:\, \eta\in D_\lambda\}\,.$$
This is a Borel sets. Indeed, the map
$(\lambda,\eta)\longmapsto\inf_{x\in\XP}(\eta-\lambda)\cdot\Phi(x)$
is upper semi-continuous, being the infimum of a family of continuous functions, while $\sP$ is Borel and the condition $\eta-\lambda\in V$ is closed. For any $a\in [0,1]$, let $G_a = G\cap (\LP\times \{F\geq a\})$. Since $F$ is Borel these are also a Borel sets. Moreover, the sections $G_a^\lambda = D_\lambda\cap\{\eta\in\LP\,:\,F(\eta)\geq a\}$ are compact, since $D_\lambda$ is compact and $F$ is upper semi-continuous. We can thus apply the Arsenin-Kunugui theorem (Theorem 18.18 in \citealt{kechris1995classical}) and conclude that the projections of $G_a$ onto their first coordinates are Borel. Since these projections coincide with $\{\lambda\in\LP\,:\,j(\lambda)\geq a\}$, we conclude that $j$ is Borel.

We now need to show that the maximisers themselves can be chosen in a Borel fashion. Define 
$$M = \{(\lambda,\eta)\in\LP\times\LP\,:\,\eta\in D_\lambda\,,\,F(\eta)=j(\lambda)\}\,.$$ This is a Borel set since $G$ is and both $F$ and $j$ are Borel. Consider now the sections $M^\lambda = \{\eta\,:\,(\lambda,\eta)\in M\} = \{\eta\in D_\lambda\,:\,F(\eta)=j(\lambda)\}$. These are non-empty because the maximum in the definition of $j$ is achieved, and they are compact since $M^\lambda =G^\lambda_{j(\lambda)}$. Applying once more the Arsenin-Kunugui theorem yields a Borel map $J:\LP\to\LP$, such that $J(\lambda)\in D_\lambda$ and $F(J(\lambda)) = j(\lambda)$ for any $\lambda\in\LP$.

We now show that $\e_{J(\lambda)}$ is maximal. Suppose, by contradiction, that there is $\eta\in\LP$ such that $\e_\eta$ strictly majorises $\e_{J(\lambda)}$. Since $\e_{J(\lambda)}$ majorises $\e_\lambda$, we have that $\e_\eta$ also majorises $e_\lambda$. So, as discussed earlier, we can find $\tilde\eta\in D_\lambda$ such that $e_{\tilde\eta} = e_\eta$. Since $\e_{\tilde\eta}$ majorises $e_\lambda$ and $\tilde\eta-J(\lambda)\in V$, we have $\tilde\eta\in D_{J(\lambda)}$. Therefore, by \eqref{eq:mcc} and since $\tilde\eta\neq J(\lambda)$ because the majorisation is strict, we have
$$F(\tilde\eta)\geq F(J(\lambda))+ r\|\tilde\eta-J(\lambda)\|>F(J(\lambda))\,,$$
which contradicts the maximality of $F(J(\lambda))$ over $D_\lambda$. So, $\e_{J(\lambda)}$ is maximal by \Cref{prop:evar}. In particular, $J(\lambda)\in\LPmax$.

Finally, suppose that $\lambda\in\LPmax$. The maximality of $e_\lambda$ implies that $\e_{J(\lambda)}=e_\lambda$. Since both $\lambda$ and $J(\lambda)$ are in $D_\lambda$, \eqref{eq:mcc} yields that
$$r\|J(\lambda)-\lambda\|\leq F(J(\lambda))-F(\lambda)=0\,,$$
Therefore $J(\lambda)=\lambda$. We also conclude that $\LPmax=\{\lambda\in\LP\,:\,J(\lambda)=\lambda\}$. Since $J$ is Borel, this set is Borel. 
\end{proof}

\section{Semi-analyticity proof}\label{app:coarse}
We let $d = \min(|\XP|,m+1)$. By \Cref{lemma:XP}, if $\HhP\neq\emptyset$ then  $d\geq 1$.

\subsection{Coarsened hypotheses}
For $T\geq 1$ (possibly, $T=\infty$), let $\Tt_T$ be the full $d$-ary tree of depth $T$ and denote as $\ROOT$ its root. For $v\in\Tt_T$, we let $|v|$ denote its depth (i.e., tree-distance from $\ROOT$). Let $\V_T = \{v\in\Tt_T\,:\,|v|>0\}$ be the set of  non-root nodes, $\I_T = \{v\in\Tt_T\,:\,|v|<T\}$ be the set of internal nodes, and $\Ell_T = \{v\in\Tt_T\,:\,|v|=T\}$ the set of leaves. For each node $v \in \V_T$, let $\pa(v)$ denote its parent. For any $v\in\I_T$, let $\ch(v)$ denote the set of its $d$ children. We also set $\ch(v)=\emptyset$ if $v\in\Ell_T$.

We define the parameter space $\Theta_T \subseteq (\XP \times [0,1])^{\V_T}$ as the set of all elements $\theta = (\hat x_v, \hat p_v)_{v \in \V_T}$ satisfying the following conditions for every   $v \in \I_T$:
\begin{enumerate}
    \item $\sum_{u \in \ch(v)} \hat p_u = 1$;
    \item $\sum_{u \in \ch(v)} \hat p_u \Phi(\hat x_u) \in S$;
    \item for every distinct $u,u'\in\ch(v)$, $\hat x_u$ and $\hat x_{u'}$ are distinct. 
\end{enumerate}

Any parameterisation $\theta \in \Theta_T$ assigns a unique location trace to each node. For any node $v \in \Tt_T$, let $\xi_\theta(v) = (\hat x_{u_1}, \dots, \hat x_{u_k}) \in \XP^{|v|}$ denote the sequence of locations along the path from the root down to $v$ (with $\xi_\theta(\ROOT)$ being the empty sequence). By construction (see condition 3 in the definition of $\Theta_T$), the mapping $v\mapsto\xi_\theta(v)$ is injective. We also define $\pi_\theta(v) = \prod_{i=1}^{|v|} \hat p_{u_i}\in[0,1]$ if $v\in\V_T$, and $\pi_\theta(\ROOT) = 1$. 

Every parameterisation $\theta \in \Theta_T$ naturally induces a probability measure $Q_\theta$ on $\XP^T$. This measure corresponds to the process that, at step $t$, transitions to a child $u \in \ch(v_{t-1})$ with conditional probability $\hat p_u$ and takes the location value $\hat x_u$. We remark that under this construction, \begin{equation}\label{eq:pi}\pi_\theta(v) = Q_\theta\left(\left\{x^T\in\XP^T\,:\,x^{|v|} = \xi_\theta(v)\right\}\right)\,.\end{equation}
We define the coarsened hypothesis $\bHhP^T$ as the set 
$$ \bHhP^T = \{Q_\theta \,:\, \theta \in \Theta_T\} \,. $$
\begin{lemma}\label{lemma:coarse}$\bHhP^T \subseteq \HhP^T$. 
\end{lemma}
\begin{proof}
To show that $\bHhP^T \subseteq \HhP^T$, take any $\theta = (\hat x_v, \hat p_v)_{v \in \V_T}\in\Theta_T$ and let $X^T\sim Q_\theta$. For any finite $t\in[1:T]$, the  support of $X^t$ can take at most $d^t$ values. In particular,  $Q_\theta\in\PR_\Phi^T$. Furthermore, for any $v$ with $\pi_\theta(v)>0$ and $|v|=t-1$, $$\E_{Q_\theta}[\Phi(X_t)|X^{t-1} = \xi_\theta(v)] = \sum_{u \in \ch(v)} \hat p_u \Phi(\hat x_u)\,.$$ Since $\theta \in \Theta_T$, this sum belongs to $S$, implying that $\E_{Q_\theta}[\Phi(X_t)|X^{t-1}] \in S$, $Q_\theta$-almost surely. We conclude that $Q_\theta \in \HhP^T$.
\end{proof}
We now verify the measurability of the parameter set $\Theta_T$.
\begin{lemma}\label{lemma:ThBor}
    If $S$ is Borel and $T<\infty$, $\Theta_T$ is a Borel subset of $(\XP \times [0,1])^{\V_T}$. 
\end{lemma}
\begin{proof}
    $\XP$ is a Borel subset of the standard Borel space $\X$ by \Cref{lemma:XP} and $\V_T$ is finite, so the ambient product space $(\XP \times [0,1])^{\V_T}$ is  standard Borel. $\Phi$ is a Borel map, so the mapping $\theta \mapsto \sum_{u \in \ch(v)} \hat p_u \Phi(\hat x_u)$ is Borel measurable for each $v \in \I_T$. Since we assume $S$ to be Borel, its pre-image under this mapping is a Borel set. Clearly, the pre-image of $\{1\}$ under $\theta\mapsto \sum_{u\in\ch(v)}\hat p_u$ is also Borel. Finally, the function that maps $\theta$ to $1$ if any two distinct children $u$ and $u'$ of $v$ satisfy $\hat x_u = \hat x_{u'}$, and $0$ otherwise, is also Borel, and its pre-image of $\{0\}$ is a Borel set. Because $T$ is finite, $\Theta_T$ is Borel, as the finite intersection of all these pre-images over the internal nodes.
\end{proof}

\subsection{Equivalence of suprema}
\begin{lemma}\label{lemma:supsup}
    Fix $T\geq 1$ and let $f:\XP^T\to[0,\infty]$ be a Borel function. Then
    $$\sup_{P\in\HhP^T}\E_P[f] = \sup_{P\in\bHhP^T}\E_P[f]\,.$$
\end{lemma}
\begin{proof}
    Without loss of generality we can assume that $f$ is bounded. Indeed, if this is not the case, we can just consider the sequence $(f^{(n)})_{n\geq1}$, where $f^{(n)} = \min(f, n)$. Once the claim is proven for the bounded case, one obtains by monotonic convergence and swapping suprema that
    $$\sup_{P\in\HhP^T}\E_P[f] =  \sup_{n\geq 1}\sup_{P\in\HhP^T}\E_P[f^{(n)}] = \sup_{n\geq 1}\sup_{P\in\bHhP^T}\E_P[f^{(n)}] = \sup_{P\in\bHhP^T}\E_P[f]\,.\footnote{Here and where necessary, we extend $f$ to be a function on $\X^T$, by setting $f(x^T) = 0$ for $x^T\in\X^T\setminus\XP^T$. Alternatively, one can show that $P(\XP^T) = 1$, which follows from \Cref{lemma:XP}, and harmlessly restrict $P$ to be a probability measure on $\XP^T$.}$$ 
    Therefore, we henceforth assume that $f$ is bounded.

    If $\HhP=\emptyset$, the statement is trivial. So, we  assume that $\HhP\neq\emptyset$. The inequality $\sup_{P\in\HhP^T}\E_P[f] \geq \sup_{P\in\bHhP^T}\E_P[f]$ follows directly from \Cref{lemma:coarse}, so we only need to prove the reverse inequality.  We now show by induction   that for any $P\in\HhP^T$ we can find a $\bar P\in\bHhP^T$ such that $\E_{\bar P}[f]\geq \E_P[f]$. 
    
    First let us consider the case $T=1$. Fix $P\in\HhP$. By \Cref{lemma:XP}, we can restrict the domain of $P$ to $\XP$. We now show that there exists $\bar P\in\bHhP$ such that $\E_{P'}[f]\geq\E_P[f]$. Let $$A = \{(\Phi(x), z)\,:\,x\in\XP\,,\;z\leq f(x)\}\subseteq\R^{m+1}\,.$$
    By \Cref{lemma:convexhull}, $(\E_P[\Phi], \E_P[f])\in \conv A$. \Cref{lemma:caradown} implies that we can find $d'\leq d=\min(m+1, |\XP|)$ distinct points $\bar x_1,\dots, \bar x_{d'}$ in $\XP$ and non-negative weights $w_1,\dots, w_{d'}$ summing to $1$ such that 
    $$\sum_{i=1}^{d'} w_i\Phi(\bar x_i) = \E_P[\Phi]\,,\qquad\sum_{i=1}^{d'} w_i f(\bar x_i) \geq \E_P[f]\,.$$ 
    In particular,  $\bar P = \sum_{i=1}^{d'} w_i \delta_{\bar x_i}\in\bHhP$ satisfies the desired requirement. 

    Let us now consider the case $T\geq 2$, assuming that the claim holds up to $T-1$. Fix $P\in\HhP^T$ and let $P'$ denote the marginal on the first  component. We let $(P_{x})_{x\in\X}$ denote a regular conditional distribution of $P$ of the components from $2$ to $T$, given the first component.  By definition of $\HhP^T$, $P_x\in\HhP^{T-1}$ holds for $P'$-almost every $x$. Without loss of generality (potentially modifying $P_x$ when $x$ is in a $P'$-null set) we can assume that $P_x\in\HhP^{T-1}$ for every $x\in\XP$. We let $g:\XP\to[0,\infty)$ be given by $$g(x) = \E_{P_x}[f(x, X^{2:T})]\,.$$ Then, $g$ is bounded (as we are assuming that $f$ is) and Borel (since we are considering a regular conditional distribution). Moreover, $\E_{P'}[g] = \E_P[f]$. 

    From what we have shown already, we know that there is $\bar P'  \in\bHhP$ such that $\E_{\bar P'}[\Phi] = \E_{P'}[\Phi]$ and $\E_{\bar P'}[g] \geq \E_{P'}[g]=\E_P[f]$. Let $\{\bar x_1,\dots, \bar x_{d'}\}\subseteq\XP$ (with $d'\leq d$) be the finite support of $\bar P'$. For each $i\in[1:d']$, we let $f_i:\X^{T-1}\to[0,\infty]$ be defined by $f_i(x^{T-1}) = f(\bar x_i, x^{T-1})$. Applying the induction hypothesis, we can find $\bar Q_1,\dots, \bar Q_{d'}$ in $\bHhP^{T-1}$ such that $$\E_{\bar Q_i}[f_i]\geq \E_{P_{\bar x_i}}[f_i]=g(\bar x_i)$$ for every $i\in[1:d']$. We now denote as $\bar P$ the discrete probability measure on $\X^{T}$ such that when $X^T\sim\bar P$, $X_1\sim \bar P'$ and $X^{2:T}$ has law $\bar Q_i$ when conditioned on $\{X_1=\bar x_i\}$. By construction, $\bar P\in\bHhP^T$ and 
    $$\E_{\bar P}[f] = \sum_{i=1}^{d'}\bar P'(\{\bar x_i\}) \E_{\bar Q_i}[f_i] \geq \E_{\bar P'}[g]\geq \E_P[f]\,.$$
    This concludes the proof. 
\end{proof}

\begin{corollary}\label{cor:supsup}
Let $f=(f_s)_{s\geq0}$ be a non-negative process, with each
$f_s:\XP^s\to[0,\infty]$ Borel. Then, for every $t\geq0$ and every $x^t\in\XP^t$,
$$\sup_{Q\in\HhP^\infty}\sup_{\tau\in\T_\star}\E_Q[f^{x^t}_\tau]
=
\sup_{Q\in\bHhP^\infty}\sup_{\tau\in\T_\star}\E_Q[f^{x^t}_\tau]\,,$$
where $f^{x^t}_s(y^s)=f_{t+s}(x^t,y^s)$.
\end{corollary}

\begin{proof}
Fix $x^t\in\XP^t$ and let $\tau\in\T_{T'}$ for some $T'\geq0$. We may regard
$\tau$ as a function on $\X^{T'}$. Define $h_\tau:\XP^{T'}\to[0,\infty]$ by
$$h_\tau(y^{T'})=f^{x^t}_{\tau(y^{T'})}(y^{\tau(y^{T'})})\,.$$
The function $h_\tau$ is Borel. Indeed, since $\tau$ is bounded by $T'$,
$$h_\tau(y^{T'})=\sum_{s=0}^{T'}\mathbf 1_{\{\tau=s\}}(y^{T'})f^{x^t}_s(y^s)\,,$$
and each set $\{\tau=s\}$ is Borel in $\X^{T'}$. By \Cref{lemma:supsup},
$$\sup_{Q\in\HhP^\infty}\E_Q[f^{x^t}_\tau] = \sup_{P\in\HhP^{T'}}\E_P[h_\tau]=\sup_{P\in\bHhP^{T'}}\E_P[h_\tau]=\sup_{Q\in\bHhP^\infty}\E_Q[f^{x^t}_\tau]\,.$$
Taking the supremum first over $\tau\in\T_{T'}$ and then over $T'\geq0$ gives the claim.
\end{proof}

\subsection{Masks}

For $T\geq 1$ finite, we define a \emph{mask} on $\Tt_T$ as a subset $M\subseteq\Tt_T$ such that: 
\begin{enumerate}
    \item[$(i)$] $\ROOT\in M$; 
    \item[$(ii)$] if $v\in M$ and $v'$ is an ancestor of $v$, then $v'\in M$; 
    \item[$(iii)$] if $v\in M$ and $|v|<T$, then either $\ch(v)\subseteq M$, or $\ch(v)\cap M = \emptyset$.
\end{enumerate}
We denote the set of all the masks on $\Tt_T$ as $\mathcal{M}_T$. We remark that $\mathcal M_T$ is a finite set. A node $v \in M$ is called a leaf of $M$ if  $\ch(v)\cap M=\emptyset$. We let $\Ell_M$ be the set of all the leaves of $M$ and define $\I_M=M\setminus\Ell_M$.

The reason we are introducing masks is that they help study  stopping times. Indeed, a pair $(\tau, \theta) \in \T_T \times \Theta_T$ naturally induces a mask $M_{\tau,\theta} \in \mathcal{M}_T$. Intuitively, the parameterisation $\theta$ ``fills'' the tree with locations, and the stopping time $\tau$ tells us where to truncate each branch based on those locations. Formally, we define $M_{\tau,\theta}$ recursively: the root $\ROOT \in M_{\tau,\theta}$, and for any internal node $v \in \I_T \cap M_{\tau,\theta}$, we include all its children $\ch(v)$ in $M_{\tau,\theta}$ if and only if $\tau(\xi_\theta(v)) > |v|$.

\begin{lemma}\label{lemma:surj}
    For $T\geq 1$ finite, the map $\Upsilon\,:\,\T_T \times \Theta_T \to \mathcal{M}_T\times\Theta_T$, given by $\Upsilon(\tau, \theta) = (M_{\tau,\theta},\theta)$, is surjective.
\end{lemma}
\begin{proof} 
    If $\Theta_T=\emptyset$, then $\T_T \times \Theta_T = \mathcal{M}_T\times\Theta_T = \emptyset$, so the claim is trivial. We thus assume that $\Theta_T\neq\emptyset$. Fix  $M \in \mathcal{M}_T$ and  $\theta \in \Theta_T$. We will construct a stopping time $\tau \in \T_T$ such that $M_{\tau,\theta} = M$. For each $t \in [0:T]$, define
    $$ \mathcal{C}_t = \left\{ \xi_\theta(v) \in \XP^t \,:\, v \in\I_M \text{ and } |v|=t \right\} \,. $$
    Since $t$ is finite, there are finitely many nodes at depth $t$, so $\mathcal{C}_t$ is finite and hence Borel. We note that the injectivity of $v\mapsto \xi_\theta(v)$ implies that, for every node $v$ at depth $t$ we have $\xi_\theta(v)\in\mathcal C_t$ if and only if $v\in \I_M$.  In particular, whether a node $v$  at depth $t$ is a leaf of $M$ can be read off from its trace $\xi_\theta(v)$.  We then define, for $x^\infty \in \X^\infty$,
    $$ \tau(x^\infty) = \inf \left\{ t \geq 0 \,:\, x^t \notin \mathcal{C}_t \right\} \,, $$ with the convention $\inf\emptyset = \infty$. For every  $t \in [0:T]$,
    $$ \{\tau \leq t\} = \bigcup_{s=0}^t \{x^\infty \in \X^\infty : x^s \notin \mathcal{C}_s\}\,, $$
    which belongs to the sigma-field generated by the first $t$ coordinates. Thus, $\tau\in\T_\infty$. Moreover, 
    since $T<\infty$, $\mathcal{C}_T = \emptyset$, which implies  $\tau \leq T$. Hence, $\tau \in \T_T$. 

    We now show that $M_{\tau,\theta}=M$. First, let us note that, for every $v\in M$, we have $\tau(\xi_\theta(v))\geq |v|$. More precisely, if $t=|v|$ and $x^\infty$ is any extension of $\xi_\theta(v)$ (i.e., $\xi_\theta(v)=x^t$), then $\tau(x^\infty)\geq t$. Indeed, let $v_0,\dots,v_t$ be the nodes on the path from $\ROOT$ to $v$, with $|v_s|=s$. Since $M$ is a mask and $v\in M$, all ancestors $v_0,\dots,v_t$ belong to $M$. Moreover, for every $s<t$, the node $v_s$ has the descendant $v$ in $M$, and therefore $v_s\in\I_M$. Hence $\xi_\theta(v_s)\in\mathcal C_s$ for every $s<t$. If $x^\infty$ is an extension of $\xi_\theta(v)$, then $x^s=\xi_\theta(v_s)$ for every $s\leq t$. Thus $x^s\in\mathcal C_s$ for every $s<t$, and by the definition of $\tau$ we obtain $\tau(x^\infty)\geq t=|v|$.

    We can now conclude that $M_{\tau,\theta}=M$ by induction on the depth. At depth $0$, both masks contain the root. Assume that the two masks coincide up to depth $t$, and let $v$ be a node of depth $t$ in this common set. By the defining property of a mask, either $\ch(v)\subseteq M$ or $v\in\Ell_M$. Using that $\tau(\xi_\theta(v))\geq t$, in the first case we must have that $\tau(\xi_\theta(v))> |v|$ and so $\ch(v)\subseteq M_{\tau,\theta}$, while in the second case we obtain $\tau(\xi_\theta(v))=|v|$ and so $\ch(v)\cap M_{\tau,\theta} = \emptyset$. Thus, the two masks coincide also at level $t+1$, which concludes the proof.
\end{proof}

\begin{comment}
\comm{Here add as corollary that if e-process is e-process for non coarsened hypothesis than it is for the coarsened one? Then discuss that all processes in the candidate form are for the coarsened one. so argue what will be our main argument... Maybe only this should stay in the main text...}
\end{comment}
\subsection{Semi-analyticity}\label{app:coarse-semian}
We use the following standard facts about upper semi-analytic functions. A function $h:\A\to[-\infty,\infty]$ is upper semi-analytic if it can be written as
$$h(a)=\sup_{b\in\B}g(a,b)$$
for some standard Borel space $\B$ and some Borel function $g:\A\times\B\to[-\infty,\infty]$.\footnote{On standard Borel spaces, this is equivalent to the usual definition of upper semi-analyticity, namely that the upper level sets $\{a\in\A:h(a)>c\}$ are analytic for every $c\in\R$.} Countable suprema and finite maxima of upper semi-analytic functions are upper semi-analytic.

\begin{proposition}\label{prop:continuation-envelope-usa}
Assume that $S$ is Borel. Let $f=(f_s)_{s\geq 1}$ be a sequence with $f_s:\XP^s\to[0,\infty]$ Borel for every $s\geq 1$. Fix $t\geq 1$. For $x^t\in\XP^t$, define $f^{x^t}=(f^{x^t}_s)_{s\geq 0}$ by $f^{x^t}_0=f_t(x^t)$ and $f^{x^t}_s(y^s)=f_{t+s}(x^t,y^s)$ for $s\geq 1$ and $y^s\in\XP^s$. Define $u:\XP^t\to[0,+\infty]$ by
$$u(x^t)=\sup_{Q\in\HhP^\infty}\sup_{\tau\in\T_\star}\E_Q[f^{x^t}_\tau]\,.$$
Then $u$ is upper semi-analytic.
\end{proposition}

\begin{proof}
    Fix $t\geq 1$. For any finite $T\geq 1$, define $u_T:\XP^t\to[0,\infty]$ as 
    $$u_T(x^t) = \sup_{\tau\in\T_T}\sup_{Q\in\bHhP^\infty}\E_{Q}[f^{x^t}_\tau]\,.$$
    
    Fix any pair $(Q, \tau)\in\bHhP^\infty\times\T_T$. Since $\tau$ is bounded by $T$, there is $\theta\in\Theta_T$ such that $\E_{Q}[f^{x^t}_\tau] = \E_{Q_\theta}[f^{x^t}_\tau]$.  So, we can rewrite the definition of $u_T$ as
    $$u_T(x^t) = \sup_{\tau\in\T_T}\sup_{\theta\in\Theta_T}\E_{Q_\theta}[f^{x^t}_\tau]\,.$$
    For any $M\in\mathcal M_T$, we define the map $h_M:\XP^t\times\Theta_T\to[0,\infty]$ as 
    $$h_M(x^t, \theta) = \sum_{v\in\Ell_M}\pi_\theta(v)f_{t+|v|}(x^t, \xi_\theta(v))\,.$$
    
    Now, fix a pair $(\tau,\theta)$, and let $M_{\tau,\theta}=\Upsilon(\tau,\theta)$ be the mask induced by the pair. By \eqref{eq:pi}, it follows that 
    $$h_{M_{\tau,\theta}}(x^t, \theta) = \E_{Q_\theta}[f^{x^t}_\tau]\,.$$ 
    In particular, thanks to the surjectivity of $\Upsilon$ from \Cref{lemma:surj}, we can rewrite the definition of $u_T$ as
    $$u_T(x^t) = \max_{M\in\mathcal M_T}\sup_{\theta\in\Theta_T}h_M(x^t, \theta)\,,$$
    where we can write a maximum since $\mathcal M_T$ is finite.

    For each $M\in\mathcal M_T$, $h_M$ is Borel. Indeed, since we are assuming that $S$ is Borel, the domain $\XP^t\times\Theta_T$  is Borel by \Cref{lemma:ThBor}, each $f_{t+|v|}$ is Borel by assumption, and both $\xi_\theta$ and $\pi_\theta$ are Borel maps. In particular, the map $x^t\mapsto \sup_{\theta\in\Theta_T}h_M(x^t, \theta)$ is upper semi-analytic, as the supremum over a Borel domain of a Borel function. We conclude that $u_T$ is upper semi-analytic as the  maximum over finitely many upper semi-analytic functions. 

    Finally, by \Cref{cor:supsup} $u=\sup_{T\geq 1}u_T$, which is upper semi-analytic as the countable supremum of upper semi-analytic functions. 
\end{proof}

\end{document}